\documentclass[letterpaper,10pt,journal,twoside]{IEEEtran}

\usepackage{cite}
\usepackage{graphicx,keyval,trig}
\usepackage{amsfonts,amssymb,dsfont}
\usepackage{mathtools,mathrsfs,amsmath}
\newenvironment{proof}{\begin{IEEEproof}}{\end{IEEEproof}}

\usepackage{enumitem}
\usepackage{float,booktabs}
\usepackage{algorithm,algorithmic}
\usepackage{xcolor}
\usepackage{hyperref}
\hypersetup{
  hidelinks,
  pdftitle={Featurized Occupation Measures for Structured Global Search in Numerical Optimal Control},
  pdfauthor={Qi Wei, Jianfeng Tao, Haoyang Tan, Hongyu Nie}
}
\usepackage{textcomp}
\usepackage[normalem]{ulem}
\usepackage{soul}
\sethlcolor{yellow}

\definecolor{reviewgreen}{rgb}{0,0.48,0.12}
\definecolor{reviewcyan}{rgb}{0,0.48,0.62}

\usepackage{subcaption}
\usepackage{preprint-theorem-wrap}
\newtheorem{theorem}{Theorem}
\newtheorem{proposition}{Proposition}
\newtheorem{definition}{Definition}
\newtheorem{corollary}{Corollary}
\newtheorem{example}{Example}
\newtheorem{remark}{Remark}
\newtheorem{lemma}{Lemma}

\usepackage{tikz}
\usetikzlibrary{patterns,arrows.meta}
\usepackage{etoolbox}
\apptocmd{\maketitle}{\thispagestyle{plain}}{}{}

\begin{document}
\title{Featurized Occupation Measures for Structured Global Search in Numerical Optimal Control}
\author{Qi Wei, Jianfeng Tao, Haoyang Tan, Hongyu Nie
\thanks{This work was supported by the National Key R\&D Program of China under Grant 2024YFF0505303.}
\thanks{Qi Wei (v7sjtu@sjtu.edu.cn), Jianfeng Tao (corresponding author, jftao@sjtu.edu.cn), Haoyang Tan (thy1998@sjtu.edu.cn), and Hongyu Nie (sapphiredrop@sjtu.edu.cn) are with the School of Mechanical Engineering, Shanghai Jiao Tong University, Shanghai 200240, China.}}

\maketitle

\begin{abstract}
Numerical optimal control is split between globally structured but dimensionally intractable Hamilton--Jacobi--Bellman (HJB) methods and scalable but local trajectory optimization. 
We introduce Featurized Occupation Measures (FOM), a finite-dimensional primal--dual interface where HJB subsolutions guide primal search and primal occupation data can inform certificate refinement. Explicit and implicit realizations, based on finite Liouville tests and rollout-sampled residuals, are proved asymptotically consistent with the exact occupation-measure LP.

FOM also identifies when HJB-type certificates avoid full state-space representation.
For factor graphs induced by compatible passivity-based interconnections, blockwise HJB inequalities yield block-additive OM-dual certificates, including blockwise approximate variants.
The curse of dimensionality is then shifted from state dimension to interconnection topology.
Approximate certificates remain reusable under time shifts and bounded model perturbations, with explicit degradation bounds.

A static obstacle-avoidance benchmark illustrates how tighter certificates improve sample-based global search.
\end{abstract}


\section{Introduction}
Optimal control governs decision-making in autonomous systems and robotics, over long horizons and under nonlinear dynamics, complex constraints, contact- or task-induced combinatorial structure \cite{betts2010practical, posa2014direct, deits2014footstep, zhang2025simultaneous}, as well as dynamic-game interactions \cite{bacsar1998dynamic,lidard2024blending, li2025multi}.
These problems are highly nonconvex; ensuring time-consistency during execution beyond nominal trajectory solutions is one of the central difficulties.

Numerical optimal control has developed along two main lines. 
The first seeks a solution of the Hamilton–Jacobi–Bellman (HJB) equation \cite{bardi1997optimal,fleming2006controlled}, whose global value function, when available, induces an optimal or near-optimal feedback law and hence closed-loop control \cite{bertsekas2012dynamic}. 
Its key advantage is the direct maintenance of global value structure, while its main limitation is the curse of dimensionality \cite{bellman1966dynamic}: global state-space discretization typically leads to complexity growing exponentially with dimension. 
Structured representations that partially circumvent this scaling include sparse and adaptive grids \cite{garcke2017suboptimal, bokanowski2013adaptive, kang2017mitigating}, low-rank tensor formats \cite{dolgov2021tensor, horowitz2014linear}, and characteristic-based pointwise evaluation \cite{darbon2016algorithms, chow2019algorithm}.
Unlike these numerical approximation approaches, a complementary classical route exploits physical structure directly. Passivity-based design \cite{sepulchre2012constructive,KOKOTOVIC2001637,freeman1996robust} systematically constructs value-like functions from subsystem interconnection topology, composing storage functions block by block, a view supported by recent results \cite{GRUNE202319,10383497}.

The other performs trajectory-wise local optimization instead of solving the global HJB equation directly. It includes first-order methods based on the Pontryagin maximum principle \cite{pontryagin1963mathematical} (or function-space KKT conditions \cite{bosch2000proof}), and second-order schemes such as DDP \cite{mayne1966second, jacobson1970differential, tassa2014control}, iLQR \cite{li2004iterative}, and SQP \cite{betts2010practical}. 
These methods exploit first- or second-order necessary conditions along a trajectory, which can be viewed, in smooth settings, as characteristic-type reductions of the HJB equation \cite{fleming1975deterministic,subbotina2006method}. Such trajectory-wise computations avoid a full state-space discretization. The cost is a tight coupling to the sequential structure of the problem, and the outcome remains an open-loop nominal solution.
Consequently, under strong nonconvexity, disconnected local minima, or information updates, such open-loop nominal solutions often have limited reusability; in the presence of uncertainty or dynamic-game structure, issues of time consistency become especially acute, as is classical in dynamic game theory \cite{bacsar1998dynamic}.

Local optimization, however, need not proceed along a single direction at each step. In difficult nonconvex problems, it is often more effective to update over a structured set of admissible candidates. Many sample-based MPC and path-integral methods \cite{kappen2005linear, theodorou2010generalized, williams2017model, williams2018information} can be interpreted in this way: each iteration samples, filters, and reweights candidate corrections from a set shaped by a search distribution or local surrogates. Learned trajectory priors can further improve search efficiency by incorporating motion structure from demonstrations and previous plans~\cite{yin2024stein,yin2025diverse}. Existing ideas from heuristic search \cite{hansen2001lao,aine2016multi,narayanan2015improved}, relaxed dynamic programming \cite{lincoln2006relaxing}, and admissible-heuristic construction for kinodynamic planning \cite{paden2016design} all suggest that globally defined value surrogates can guide local search.
We seek a framework that combines trajectory-wise efficiency with a persistent global value structure.

The occupation-measure (OM) formulation connects these two viewpoints. 
It lifts nonlinear optimal control to an infinite-dimensional linear program over measures, whose dual is an HJB inequality \cite{vinter1993convex,lasserre2008nonlinear,kamoutsi2017infinite}. 
Dual-feasible functions are HJB subsolutions, i.e., globally valid lower-bound certificates. 
They can be used to compare candidate trajectories, prune prefixes, and identify where a search should be refined \cite{paden2016design,lincoln2006relaxing,brown2022information,salzman2015asymptotically}.
HJB subsolutions can also be optimized through semi-infinite programming~\cite{oustry2023minimal}.

Finite-dimensional realizations of this primal--dual picture take different forms. 
Moment--SOS hierarchies approximate the OM primal--dual problem for polynomial data through convex moment relaxations and SOS positivity certificates \cite{lasserre2008nonlinear}. 
Sample-based numerical optimal control methods generate trajectories or trajectory distributions that induce occupation measures. 
Their updates are usually specified directly in trajectory or parameter space, leaving the associated OM-dual certificate layer implicit. 
This leaves open a finite-dimensional interface in which the global certificate is explicit and the primal admits more general realizations (not necessarily a convex relaxation).

We introduce featurized occupation measures (FOM) as a certificate-first formulation of this interface. 
A FOM realization maintains an induced occupation-type primal pair, a finite-dimensional HJB-certificate class, and a Liouville residual. 
The primal realization need not be a convex finite-dimensional OM relaxation; it may arise from weak-form approximation, or sampling-based search. 
The Liouville residual records the departure from exact OM feasibility, allowing approximate certificates to retain quantitative meaning as lower bounds, gap indicators, and guides for pruning or refinement. 
The primal realization can in turn guide certificate refinement through its occupation data.
Thus rollout and sample-based updates, including MPPI \cite{kappen2005linear,theodorou2010generalized,williams2017model,williams2018information} and RL-style policy search methods \cite{schulman2017proximal,silver2014deterministic,varnai2020path,cai2026qguidedsteinvariationalmodel}, can be equipped with an explicit OM-dual certificate layer. 
Pruning, guidance, and warm starts are then expressed at the global level; see also \cite{holtorf2024bounds}.

The certificate-first viewpoint makes it natural to ask whether an OM-dual certificate can be decomposed, following the same structural logic as passivity-based storage functions. 
Whether this parallel holds is a central concern of this paper. 
The certificate would then be more than an a posteriori bound. It would be a structured object whose tractability can be designed through systematic exploitation of task structure.

\emph{Contributions.}
\begin{enumerate}
    \item \textbf{FOM framework.} We formulate Featurized Occupation Measures as a finite-dimensional primal--dual interface for the occupation-measure LP relaxation of Bolza optimal control that pairs a numerical optimal control solver with an explicit HJB subsolution, so that global certificates guide primal search in a primal-dual language.

    \item \textbf{Explicit and implicit FOM.} A gap decomposition (Proposition~\ref{prop:fom_lower_bound_interface}) extends the OM framework to arbitrary primal realizations and approximate certificates. This decoupling is what allows nonconvex primal realization. We give two complementary realizations, explicit and implicit FOM, and prove their asymptotic consistency with the exact OM LP (Theorems~\ref{thm:explicit_fom_consistency} and~\ref{thm:implicit_fom_consistency}), with primal expressiveness and dynamical feasibility treated separately.
    
    \item \textbf{Structural certificate complexity.}
    Under compatible cost and dynamics decompositions and the stated feasibility conditions, the factor graph encoding interconnection topology is proved structurally invariant: the same block decomposition that supports a passivity-based design also supports a block-additive OM-dual certificate (Lemma~\ref{lem:block_hjb_slack_factorization}, Theorem~\ref{thm:blockwise_om_dual_assembly}, Proposition~\ref{prop:inherited_factor_graph_ioc}).
    We prove that per-block approximation errors propagate additively (Corollary~\ref{cor:properties_approx_block_additive}), shifting the curse of dimensionality from state-space dimension to interconnection topology. Certificates remain valid across time shifts and bounded model perturbations (Proposition~\ref{prop:properties_certificate_reuse}), making them reusable as warm starts.
\end{enumerate}
The paper is organized as follows. Section~\ref{sec:preliminaries} reviews the OM primal--dual structure. Section~\ref{sec:fom} formulates FOM and its mechanisms, with explicit and implicit FOM realizations, and proves their asymptotic consistency with the exact OM LP. Section~\ref{sec:properties} develops structural certificate results and computational consequences. Section~\ref{sec:experiments} demonstrates the framework on a benchmark problem, and Section~\ref{sec:conclusion} concludes.

\section{Preliminaries}
\label{sec:preliminaries}
\subsection{From Bolza optimal control to occupation measures}

Consider the deterministic fixed-horizon Bolza optimal control problem
\begin{equation}
\label{eq:BolzaProblem}
\begin{aligned}
\min_{x(\cdot),u(\cdot)}\;
&\int_{t_0}^{T}\ell\big(t,x(t),u(t)\big)\,dt \;+\; g\big(x(T)\big)\\
\mathrm{s.t.}\quad
&\dot x(t)=f\big(t,x(t),u(t)\big),
\qquad t\in[t_0,T],\\
&x(t_0)=x_0,\quad
u(t)\in\mathcal U,\; x(t)\in X,
\end{aligned}
\end{equation}
where $X\subseteq\mathbb R^n$ is the state space, $\mathcal U\subseteq\mathbb R^m$ is the admissible control set, $\ell:[t_0,T]\times X\times\mathcal U\to\mathbb R$ is the running cost, $g:X\to\mathbb R$ is the terminal cost, and $f:[t_0,T]\times X\times\mathcal U\to\mathbb R^n$ is the vector field.
A pair $(x(\cdot),u(\cdot))$ satisfying the dynamics and the stated constraints is an admissible state--control pair.

The occupation-measure lifting reformulates this nonlinear, nonconvex problem as an infinite-dimensional linear program over measures. 
Define the support set
\[
Z:=[t_0,T]\times X\times\mathcal U,
\]
and let $C(Z)$ be the space of continuous functions on $Z$, $C^1([t_0,T]\times X)$ the space of continuously differentiable functions on the time--state domain.
Denote the dual pairing between a measure and a continuous function by $\langle \phi,\rho\rangle:=\int \phi\,d\rho$.

If state constraints are present, one restricts the state to $X_{\mathrm{adm}}\subset X$ and replaces $Z$ with $Z_{\mathrm{adm}}:=[t_0,T]\times X_{\mathrm{adm}}\times\mathcal U$.

Given an admissible pair $(x(\cdot),u(\cdot))$, its induced \emph{occupation measure} $\mu\in\mathcal M_+(Z)$ is defined by
\begin{equation}
\label{eq:omOccupationMeasure}
\langle \varphi,\mu\rangle
\ :=\
\int_{t_0}^{T}\varphi\big(t,x(t),u(t)\big)\,dt,
\qquad \forall \varphi\in C(Z).
\end{equation}
Equivalently, for every Borel set $A\subset Z$,
\begin{equation}
\label{eq:omOccupationMeasureSet}
\mu(A)
\ :=\
\int_{t_0}^{T}\mathbf 1_A\big(t,x(t),u(t)\big)\,dt;
\end{equation}
$\mu$ is the pushforward of the Lebesgue measure on $[t_0,T]$ under the map
$t\mapsto (t,x(t),u(t))$.

Let $\mu_0\in\mathcal M_+(\{t_0\}\times X)$ and $\mu_T\in\mathcal M_+(\{T\}\times X)$ be the initial and terminal boundary measures.
We write $\mu_0(X)$ and $\mu_T(X)$ for the total masses.
In the single deterministic trajectory case, $\mu_0=\delta_{(t_0,x_0)}$ and $\mu_T=\delta_{(T,x(T))}$.

The Bolza cost becomes a linear functional in $(\mu,\mu_T)$:
\begin{equation}
\label{eq:omLinearBolzaCost}
\int_{t_0}^{T}\ell\big(t,x(t),u(t)\big)\,dt+g\big(x(T)\big)
=
\langle \ell,\mu\rangle+\langle g,\mu_T\rangle .
\end{equation}

Define the space--time transport operator $L_f:C^1([t_0,T]\times X) \to C(Z)$
\begin{equation}
\label{eq:LfDefinition}
(L_fv)(t,x,u) :=\partial_t v(t,x)+\nabla_x v(t,x)^\top f(t,x,u).
\end{equation}
For any $v\in C^1([t_0,T]\times X)$, the chain rule yields $\frac{d}{dt}v(t,x(t)) = L_fv(t,x(t),u(t))$.
Integrating from $t_0$ to $T$ and rewriting with the occupation measure yields the \emph{weak Liouville relation}
\begin{equation}
\label{eq:liouvilleConstraint}
\langle v,\mu_T-\mu_0\rangle
=
\langle L_f v,\mu\rangle,
\qquad
\forall v\in C^1([t_0,T]\times X).
\end{equation}
In the Dirac case $\mu_0=\delta_{(t_0,x_0)}$, $\mu_T=\delta_{(T,x(T))}$, this reduces to the trajectory-wise transport identity
\begin{equation}
\label{eq:liouvilleTransportIdentity}
v(T,x(T))-v(t_0,x_0)
=
\int_{t_0}^{T}L_fv(t,x(t),u(t))\,dt.
\end{equation}
Since $L_f$ carries $f$, the weak Liouville relation expresses the nonlinear dynamics as a linear equation in measures.

Collecting the linearized cost~\eqref{eq:omLinearBolzaCost} and the Liouville constraint~\eqref{eq:liouvilleConstraint} yields the
\emph{occupation measure linear program (OM LP)}:
\begin{equation}
\label{eq:omPrimalLp}
\begin{aligned}
P:=&\inf_{\mu\in\mathcal M_+(Z),\ \mu_T\in\mathcal M_+(\{T\}\times X)}
\langle \ell,\mu\rangle+\langle g,\mu_T\rangle\\
&\mathrm{s.t.}\quad
\langle v,\mu_T-\mu_0\rangle = \langle L_f v,\mu\rangle,
\; \forall v\in C^1([t_0,T]\times X).
\end{aligned}
\end{equation}

The OM LP is a convex relaxation of the Bolza problem.
Under the classical local convexity condition that, for every $(t,x)$, $f(t,x,\mathcal U)$ is convex and
\[
    v\mapsto \inf\{\ell(t,x,u):u\in\mathcal U,\ f(t,x,u)=v\}
\]
is convex, the OM LP and the original Bolza problem have the same value~\cite{lasserre2008nonlinear}. 
This covers control-affine dynamics $\dot x=a(t,x)+B(t,x)u$ with convex $\mathcal U$ and $\ell$ convex in $u$. 
Otherwise, the OM LP is a relaxed-control convexification and may yield a strictly smaller value.

\subsection{OM dual problem and the HJB slack}

The dual variable of the OM LP~\eqref{eq:omPrimalLp} is a smooth function $v\in C^1([t_0,T]\times X)$, and the dual problem is
\begin{equation}
\label{eq:omDualLp}
\begin{aligned}
P^\star:=&\sup_{v\in C^1([t_0,T]\times X)}
\langle v,\mu_0\rangle\\
&\mathrm{s.t.}\quad
\ell(t,x,u)+(L_fv)(t,x,u)\ge 0,
\quad \forall (t,x,u)\in Z,\\
&\quad \;\; g(x)-v(T,x)\ge 0,
\quad \forall x\in X.
\end{aligned}
\end{equation}
The dual constraints are the HJB inequalities.

\begin{definition}[HJB slack]
\label{def:hjb_slacks}
For $v\in C^1([t_0,T]\times X)$, the \emph{HJB slack} is the pair $(s,s_T)$ of running and terminal slacks
\begin{equation}
\label{eq:hjbSlacks}
\begin{aligned}
s(t,x,u)&:=\ell(t,x,u)+L_fv(t,x,u),\qquad (t,x,u)\in Z,\\[2pt]
s_T(x)&:=g(x)-v(T,x),\qquad x\in X.
\end{aligned}
\end{equation}
\end{definition}

Dual feasibility is then
\[
s(t,x,u)\ge 0\quad(\forall (t,x,u)\in Z),\qquad
s_T(x)\ge 0\quad(\forall x\in X).
\]
A dual-feasible $v$ is an \emph{HJB subsolution} or an \emph{OM-dual certificate}.
Expanded, the dual constraints read
$-\partial_t v(t,x)\le \inf_{u\in\mathcal U}\{\ell(t,x,u)+\nabla_x v(t,x)^\top f(t,x,u)\}$,
$v(T,x)\le g(x)$.

For any primal-feasible $(\mu,\mu_T)$ and dual-feasible $v$, weak duality yields
\[
\langle v,\mu_0\rangle
\le
\langle \ell,\mu\rangle+\langle g,\mu_T\rangle.
\]
Using the slacks, the primal--dual gap decomposes as
\begin{equation}
\label{eq:omPrimalDualGap}
\begin{aligned}
\operatorname{gap}(\mu,\mu_T;v)
&:=
\big[\langle \ell,\mu\rangle+\langle g,\mu_T\rangle\big]-\langle v,\mu_0\rangle\\
&\;=
\langle \ell+L_fv,\mu\rangle+\langle g-v(T,\cdot),\mu_T\rangle\\
&\;=
\langle s,\mu\rangle+\langle s_T,\mu_T\rangle
\;\ge\; 0.
\end{aligned}
\end{equation}
The second equality follows from substituting the weak Liouville relation $\langle v,\mu_0\rangle = \langle v,\mu_T\rangle-\langle L_fv,\mu\rangle$ for the primal-feasible pair $(\mu,\mu_T)$.
Section~\ref{sec:fom} extends this identity to finitely realized primal pairs and approximate dual certificates.

When $X,\mathcal U$ are compact and $\ell$, $g$, and $L_f v$ (for all $v\in C^1$) 
are continuous on $Z$, the OM LP satisfies strong duality~\cite{lasserre2008nonlinear}:
\[
P=P^\star.
\]

\begin{remark}[$v$ as a weak transport test and certificate]
\label{rem:probe_certificate}
The dual variable $v$ serves as both a transport probe and a cost certificate.
Through the weak Liouville relation, it tests the dynamical consistency of the measure pair.
With the dual feasibility $s\ge0$, $s_T\ge0$, $v$ provides a global lower bound on the Bolza cost:
\[
v(t_0,x_0) \le \int_{t_0}^{T}\ell\,dt + g(x(T))
\]
along every admissible trajectory.
\end{remark}

\subsection{Approximate dual feasibility and relaxed certificate bounds}
\label{subsec:prelim_eps_certificate}

Exact dual feasibility requires $s(t,x,u)\ge0$ at every $(t,x,u)\in Z$.
A certificate from a finite-parameter family $\mathcal V_\Psi$ may violate exact feasibility at some $(t,x,u)$; the FOM framework should tolerate bounded violations and correct the bound accordingly.
Definition~\ref{def:approx_dual_feasibility} relaxes the pointwise condition and yields a corrected lower bound.

\begin{definition}[$(\varepsilon,\varepsilon_T)$-feasibility]
\label{def:approx_dual_feasibility}
For $\varepsilon,\varepsilon_T\ge 0$, $v$ is \emph{$(\varepsilon,\varepsilon_T)$-feasible} if
\begin{equation}
\label{eq:approxDualFeasible}
\begin{aligned}
s(t,x,u)&\ge -\varepsilon\quad(\forall (t,x,u)\in Z),
\\
s_T(x)&\ge -\varepsilon_T\quad(\forall x\in X).
\end{aligned}
\end{equation}
\end{definition}
Since $s\ge -\varepsilon$ is equivalent to $(\ell+\varepsilon)+L_fv\ge0$, an $(\varepsilon,\varepsilon_T)$-feasible $v$ is an HJB subsolution for the perturbed cost $(\ell+\varepsilon,\;g+\varepsilon_T)$. 

\begin{proposition}[Certified lower bounds from approximate certificates]
\label{prop:prelim_approx_OM_lower_bound}
In the deterministic fixed-horizon setting, every primal-feasible pair satisfies
$\mu(Z)=(T-t_0)\,\mu_0(X)$ and $\mu_T(X)=\mu_0(X)$.
Let $v$ be $(\varepsilon,\varepsilon_T)$-feasible and define
\[
\underline P(v)
:= \langle v(t_0,\cdot),\mu_0\rangle - (T-t_0)\,\mu_0(X)\,\varepsilon - \mu_0(X)\,\varepsilon_T,
\]
Then $\underline P(v)\le P$.
\end{proposition}

\begin{proof}
From the gap decomposition~\eqref{eq:omPrimalDualGap} and $s\ge -\varepsilon$, $s_T\ge -\varepsilon_T$,
\[
\begin{aligned}
\langle \ell,\mu\rangle+\langle g,\mu_T\rangle-\langle v,\mu_0\rangle
&= \langle s,\mu\rangle+\langle s_T,\mu_T\rangle\\
&\ge -\varepsilon\,\mu(Z)-\varepsilon_T\,\mu_T(X).
\end{aligned}
\]
Rearranging and substituting $\mu(Z)=(T-t_0)\,\mu_0(X)$, $\mu_T(X)=\mu_0(X)$ yields, for every primal-feasible pair,
\[
\langle v(t_0,\cdot),\mu_0\rangle - (T-t_0)\,\mu_0(X)\,\varepsilon - \mu_0(X)\,\varepsilon_T
\le \langle \ell,\mu\rangle+\langle g,\mu_T\rangle .
\]
The left-hand side is $\underline P(v)$. Taking the infimum over all primal-feasible pairs yields $\underline P(v)\le P$.
\end{proof}

Proposition~\ref{prop:prelim_approx_OM_lower_bound} establishes that an approximate certificate yields a lower bound, with degradation linear in $\varepsilon$ and $\varepsilon_T$, and is independent of the state-space dimension.
For a single trajectory with $\mu_0=\delta_{x_0}$, the bound reduces to
$v(t_0,x_0) - \varepsilon(T-t_0) - \varepsilon_T \le P$.

\section{Featurized Occupation Measures}
\label{sec:fom}
This section studies \emph{finite-dimensional} realizations of the exact OM LP.
Under the certificate-first viewpoint, the goal is to retain the OM primal--dual structure at finite resolution while allowing flexible, possibly nonconvex primal realizations.
We focus on two such primal realizations.

In an \emph{explicit FOM}, one parameterizes a global primal pair and represents the Liouville residual through finitely many test functions, which can be computationally expensive.
In an \emph{implicit FOM}, one instead uses existing ODE integrators to generate segmentwise or full rollouts, as in multiple or single shooting, and induces occupation measures by integrating Dirac masses along rollouts.
The resulting finite models are different, but both will be shown asymptotically consistent with the exact OM LP under the approximation conditions stated below.
Set
\[
\begin{aligned}
Z &:=[t_0,T]\times X\times \mathcal U,\\
\mathcal M &:=\mathcal M_+(Z)\times \mathcal M_+(\{T\}\times X),\\
\mathcal V &:=C^1([t_0,T]\times X).
\end{aligned}
\]

\subsection{General FOM components}
\label{subsec:general_fom_structure}
A featurized occupation measure (FOM) is any finite-dimensional realization that
(i) produces a global primal pair,
(ii) keeps a finite-capacity certificate class explicit, and
(iii) equips the realized primal pair with a residual mechanism measuring departure from Liouville feasibility.

Fix a finite-capacity trial family $\Theta^{(M)}$ parameterized by $\theta$, and a finite-dimensional certificate family $\mathcal V_\Psi$ parameterized by $\psi$.
\begin{definition}[FOM framework]
\label{def:fom_framework}
A general FOM realization consists of the following objects.

First, a \emph{trial realization family}
\[
\Theta^{(M)},
\]
whose elements $\theta$ are finite-dimensional parameters.
In deterministic optimal control, typical examples include an open-loop control $u_\theta:[t_0,T]\to \mathcal U$ and a feedback law
$\kappa_\theta:[t_0,T]\times X\to \mathcal U$.
More generally, $\theta$ need only parameterize a realizable control mechanism that induces a global primal pair.

Second, an \emph{induced global primal pair}
\[
\theta \longmapsto (\mu_\theta,\mu_{T,\theta})\in \mathcal M,
\]
where $\mu_\theta$ is the realized occupation-type measure, $\mu_{T,\theta}$ the realized terminal measure.
The realized cost is then evaluated by
\[
J(\theta)
:=
\langle \ell,\mu_\theta\rangle
+
\langle g,\mu_{T,\theta}\rangle .
\]

Third, a parameterized finite-capacity \emph{certificate class}
\[
\mathcal V_\Psi=\{v_\psi:\psi\in\Psi\}\subset\mathcal V,
\]
which preserves the dual side of the OM formulation at finite dimension.

Fourth, for each realized primal pair $(\mu_\theta,\mu_{T,\theta})$, we associate its
\emph{Liouville residual} (or primal-feasibility residual) functional
\begin{equation}
\label{eq:fomLiouvilleResidual}
\mathcal R_\theta(v)
:=
\langle v,\mu_{T,\theta}-\mu_0\rangle
-
\langle L_f v,\mu_\theta\rangle,
\quad
v\in\mathcal V.
\end{equation}
Exact Liouville feasibility is equivalent to
\[
\mathcal R_\theta(v)=0
\qquad
\forall v\in\mathcal V.
\]
\end{definition}

Every Liouvillian pairing $\mu\mapsto\langle L_f v,\mu\rangle$ that appears in a residual is assumed well-defined and weak-$\star$ continuous on the induced measure class.
The condition is automatic when $f$ is continuous and measures have compact support. It also holds for bounded piecewise-continuous $f$ whose jump set carries zero mass under the induced occupation measures, including their limits under weak-$\star$ convergence.

\subsection{Certificate evaluation and residual-aware bounds}
\label{subsec:fom_certificate_mechanisms}
In the FOM setting, the realized trial pairs $(\mu_\theta,\mu_{T,\theta})$ may violate Liouville feasibility.
The residual $\mathcal R_\theta(v_\psi)$ corrects the lower bound, giving the following residual-aware gap decomposition.

Following Section~\ref{sec:preliminaries}, the HJB slack of certificate $v_\psi$ is the pair $(s_\psi,s_{T,\psi})$ given by
\begin{equation}
\label{eq:fomSlacks}
\begin{aligned}
s_\psi(t,x,u)&:=\ell(t,x,u)+(L_f v_\psi)(t,x,u),\\
s_{T,\psi}(x)&:=g(x)-v_\psi(T,x).
\end{aligned}
\end{equation}
Recall from Definition~\ref{def:approx_dual_feasibility} that $v_\psi$ is $(\varepsilon,\varepsilon_T)$-feasible iff
\[
s_\psi(t,x,u)\ge -\varepsilon,\ \forall (t,x,u)\in Z,
\quad
s_{T,\psi}(x)\ge -\varepsilon_T,\ \forall x\in X.
\]

Generalizing Proposition~\ref{prop:prelim_approx_OM_lower_bound} to realized primal pairs $(\mu_\theta,\mu_{T,\theta})$, define the residual-corrected lower bound on the realized cost
\[
\underline J(\theta,\psi)
:=
\langle v_\psi,\mu_0\rangle
-\varepsilon\,\mu_\theta(Z)
-\varepsilon_T\,\mu_{T,\theta}(X)
-|\mathcal R_\theta(v_\psi)|
\]
and the corresponding certificate-relative gap
\[
\mathrm{Gap}(\theta,\psi)
:=
J(\theta)-\underline J(\theta,\psi).
\]

\begin{proposition}[Residual-aware gap decomposition]
\label{prop:fom_lower_bound_interface}
Let $v_\psi$ be $(\varepsilon,\varepsilon_T)$-feasible.
Then, for every trial parameter $\theta$ from any finite-dimensional realization class $\Theta^{(M)}$,
\begin{equation}
\label{eq:fomPrimalDualIdentity}
J(\theta)
=
\langle v_\psi,\mu_0\rangle
+\langle s_\psi,\mu_\theta\rangle
+\langle s_{T,\psi},\mu_{T,\theta}\rangle
+\mathcal R_\theta(v_\psi),
\end{equation}
hence
\[
\underline J(\theta,\psi)\le J(\theta),
\]
and
\begin{equation}
\label{eq:fomGapDecomposition}
\begin{aligned}
\mathrm{Gap}(\theta,\psi)
=\ &
\langle s_\psi+\varepsilon,\mu_\theta\rangle
+\langle s_{T,\psi}+\varepsilon_T,\mu_{T,\theta}\rangle\\
&+2[\mathcal R_\theta(v_\psi)]_+
\;\ge\;0,
\end{aligned}
\end{equation}
where $[a]_+:=\max\{a,0\}$.
\end{proposition}

\begin{proof}
By definition,
\[
J(\theta)=\langle \ell,\mu_\theta\rangle+\langle g,\mu_{T,\theta}\rangle.
\]
Adding and subtracting $\langle L_fv_\psi,\mu_\theta\rangle$ and
$\langle v_\psi(T,\cdot),\mu_{T,\theta}\rangle$, and using the definition of
$\mathcal R_\theta(v_\psi)$, gives \eqref{eq:fomPrimalDualIdentity}.
Since $v_\psi$ is $(\varepsilon,\varepsilon_T)$-feasible,
\[
\langle s_\psi,\mu_\theta\rangle\ge -\varepsilon\,\mu_\theta(Z),
\qquad
\langle s_{T,\psi},\mu_{T,\theta}\rangle\ge -\varepsilon_T\,\mu_{T,\theta}(X).
\]
Substituting these bounds into \eqref{eq:fomPrimalDualIdentity} yields
$\underline J(\theta,\psi)\le J(\theta)$.
Subtracting $\underline J(\theta,\psi)$ from \eqref{eq:fomPrimalDualIdentity} gives
\[
\begin{aligned}
\mathrm{Gap}(\theta,\psi)
=\ &
\langle s_\psi+\varepsilon,\mu_\theta\rangle
+\langle s_{T,\psi}+\varepsilon_T,\mu_{T,\theta}\rangle\\
&+\mathcal R_\theta(v_\psi)+|\mathcal R_\theta(v_\psi)|,
\end{aligned}
\]
which is \eqref{eq:fomGapDecomposition} because
$|a|+a=2[a]_+$.
\end{proof}

The gap decomposition~\eqref{eq:fomGapDecomposition} separates certificate-relative slack from transport inconsistency.
For an exactly Liouville-feasible primal pair, the corrected gap is the sum of the nonnegative running and terminal shifted slacks.
A small gap therefore controls accumulated slack and motivates certificate-guided restriction of the primal search.

\subsubsection*{Mechanism (a) Certificate-guided primal search}

A dual-feasible certificate induces an admissibility structure for primal search by providing lower bounds on the total cost of feasible completions of a given prefix.
If such a bound exceeds the cost of a known feasible solution with the same initial measure, the prefix cannot improve on that solution and may be pruned.
We first state this rule for exactly Liouville-feasible primal pairs.

\begin{corollary}[Cumulative lower-bound pruning]
\label{cor:cumulative_lower_bound_pruning}
Let $v_\psi$ be $(\varepsilon,\varepsilon_T)$-feasible and let
$(\mu,\mu_T)$ satisfy the weak Liouville equation.
At time $t$, write $\mu=\mu_{\le t}+\mu_{>t}$ and let $\mu_t$
be the induced state marginal. Define
\[
\begin{aligned}
B_\psi(\mu_{\le t})
:=\ &\langle\ell,\mu_{\le t}\rangle
+\langle v_\psi(t,\cdot),\mu_t\rangle\\
&-\bigl[\varepsilon(T-t)+\varepsilon_T\bigr]\mu_t(X).
\end{aligned}
\]
Then every feasible completion of this prefix satisfies
\[
\langle\ell,\mu\rangle+\langle g,\mu_T\rangle
\ge B_\psi(\mu_{\le t}).
\]
Consequently, if $B_\psi(\mu_{\le t})$ exceeds the cost of
a known feasible pair with the same initial measure,
the prefix cannot improve on that pair and may be pruned.
\end{corollary}

\begin{proof}
Proposition~\ref{prop:prelim_approx_OM_lower_bound}, applied
on $[t,T]$ with initial measure $\mu_t$, bounds the suffix
cost below by
\[
\langle v_\psi(t,\cdot),\mu_t\rangle
-\bigl[\varepsilon(T-t)+\varepsilon_T\bigr]\mu_t(X).
\]
Adding the prefix cost gives the claim.
\end{proof}

For a deterministic feasible trajectory starting from $(t_0,x_0)$, let $J_{\mathrm{pre}}(t)$ denote the running cost accumulated along its prefix up to $t$. The prefix bound becomes
\begin{equation}
\label{eq:fomPrefixBound}
B_\psi(t)
:=J_{\mathrm{pre}}(t)+v_\psi(t,x(t))
-\varepsilon(T-t)-\varepsilon_T.
\end{equation}

Let $\bar J$ be the cost of a feasible incumbent trajectory from the same initial condition, and let $\tau_{\mathrm{tol}}\ge0$ be a search tolerance.
Define the certificate-induced admissibility envelope
\begin{equation}
\label{eq:fomAdmissiblePrefixes}
\mathcal P_\psi^{\tau_{\mathrm{tol}}}(\bar J)
:=
\left\{
\gamma_{\le t}:
B_\psi(t)\le\bar J+\tau_{\mathrm{tol}}
\right\},
\end{equation}
where $\gamma_{\le t}$ ranges over feasible trajectory prefixes from $(t_0,x_0)$.
Membership in this envelope means that a prefix has not been excluded by the lower-bound test.
Every feasible trajectory of cost at most $\bar J$ has all its prefixes in this envelope.

To relate this search restriction to the dual slack, set $\widehat s_\psi:=s_\psi+\varepsilon\ge0$.
The chain rule gives
\begin{equation}
\label{eq:fomCumulativeSlack}
B_\psi(t)
=
B_\psi(t_0)
+\int_{t_0}^{t}
\widehat s_\psi(\tau,x(\tau),u(\tau))\,d\tau.
\end{equation}

By~\eqref{eq:fomCumulativeSlack}, the admissibility test is
equivalently
\begin{equation}
\label{eq:fomSlackBudget}
\int_{t_0}^{t}\widehat s_\psi\,d\tau
\le
\bar J+\tau_{\mathrm{tol}}-B_\psi(t_0).
\end{equation}
Thus the incumbent and the initial certified lower bound
determine a budget for accumulated dual slack.
When $\bar J=B_\psi(t_0)$ and $\tau_{\mathrm{tol}}=0$, every retained prefix accumulates zero shifted slack, consistent with complementary slackness along optimal occupation measures.
A positive budget permits bounded cumulative slack, giving the admissibility envelope a dual active-set interpretation at finite gaps.

\begin{remark}[Dual certificates and heuristic search]
\label{rem:incumbent_specialization}
The pruning rule compares a lower bound with a feasible incumbent cost, as in branch-and-bound search.
For a shortest-path problem with edge costs $c_{ij}$, a consistent heuristic $h$, normalized to zero at the goal, satisfies
\[
c_{ij}+h(j)-h(i)\ge0.
\]
These are the feasibility inequalities for a shortest-path dual potential, with reduced edge costs $r_{ij}:=c_{ij}+h(j)-h(i)$.
Along a path prefix from $i_0$ to $i$,
\[
g_{\mathrm{pre}}+h(i)
=
h(i_0)+\sum_{(j,k)\,\text{in prefix}}r_{jk}.
\]
The A$^\ast$ priority therefore equals the initial dual potential plus accumulated dual slack~\cite{hart1968formal}.
An incumbent-based threshold on this priority imposes a reduced-cost budget on candidate prefixes.
The corrected certificate $v_\psi(t,x)-\varepsilon(T-t)-\varepsilon_T$ supplies the admissible cost-to-go term in~\eqref{eq:fomPrefixBound}~\cite{paden2016design}.
For feasible trajectory segments connecting time--state nodes, let $c_{ij}:=\int_{t_i}^{t_j}\ell(t,x(t),u(t))\,dt$ denote the segment cost. Equation~\eqref{eq:fomCumulativeSlack} then identifies reduced cost with integrated shifted HJB slack. In FOM, \eqref{eq:fomAdmissiblePrefixes} governs prefix pruning while the primal solver generates candidates. This division of roles parallels the consistent anchor and auxiliary heuristics in multi-heuristic search~\cite{aine2016multi,narayanan2015improved}.
\end{remark}
\subsubsection*{Mechanism (b) Primal-informed certificate refinement}

The primal realization can in turn guide certificate refinement.
Its occupation measure identifies regions relevant to the current search and can guide approximation weights, basis refinement, or the order of Bellman and constraint updates.
These choices apply to both explicit and implicit FOM; the resulting certificate must still satisfy the global feasibility bounds.

A direct realization uses a probability measure $\nu_\theta$ on $[t_0,T]\times X$ to weight certificate accuracy.
For a nonzero occupation measure, one choice is its normalized time--state marginal,
\[
\nu_\theta(dt,dx)
:=\frac{\mu_\theta(dt,dx,\mathcal U)}{\mu_\theta(Z)}.
\]
The weights can be obtained from an explicit measure representation or approximated by rollout samples, without requiring exact Liouville feasibility.
Assume that the remaining optimal cost is finite and integrable under $\nu_\theta$.
For fixed $(\varepsilon,\varepsilon_T)$, set
\[
\underline v_\psi(t,x)
:=v_\psi(t,x)-\varepsilon(T-t)-\varepsilon_T,
\]
and hold $\nu_\theta$ fixed during the update
\begin{equation}
\label{eq:fomDualUpdate}
\begin{aligned}
    \max_\psi\quad &
    \langle \underline v_\psi,\nu_\theta\rangle \\
    \text{s.t.}\quad &
    s_\psi(t,x,u)\ge -\varepsilon,
    \qquad \forall (t,x,u)\in Z,\\
    &
    s_{T,\psi}(x)\ge -\varepsilon_T,
    \qquad \forall x\in X.
\end{aligned}
\end{equation}

The reference measure selects where the finite-capacity certificate is tightened~\cite{de2003linear,paden2016design}.
An accepted update preserves feasibility and does not decrease this objective.
Individual lower bounds may decrease even when their weighted average improves; if needed, the initial bound $\langle\underline v_\psi,\mu_0\rangle$ can be constrained not to decrease.

For a feasible rollout suffix with cost $J_{\mathrm{suf}}(t)$, the gap identity gives
\[
\begin{aligned}
J_{\mathrm{suf}}(t)-\underline v_\psi(t,x(t))
&=\int_t^T\widehat s_\psi(\tau,x(\tau),u(\tau))\,d\tau\\
&\quad+s_{T,\psi}(x(T))+\varepsilon_T.
\end{aligned}
\]
With the suffixes fixed and their starting points weighted by $\nu_\theta$, the update minimizes their average certificate-relative gap.
The updated cost-to-go bounds enter the prefix test~\eqref{eq:fomPrefixBound} in Mechanism (a).

A fixed reference measure $\nu_\theta\equiv\nu_{\mathrm{ref}}$ is a special case.
For example, uniform weights on a prescribed time--state grid give equal importance to lower-bound accuracy at each node.
With independent nodal values and a monotone Bellman discretization, the discrete counterpart of~\eqref{eq:fomDualUpdate} can be solved by semi-Lagrangian dynamic programming~\cite{falcone2014semi}.
The reconstructed continuous function must still satisfy the stated running and terminal feasibility bounds.

\begin{remark}[Dynamic programming and exact certificates]
    \label{rem:dp_zero_gap}
    When the classical value function is available, it satisfies the HJB equation and the terminal equality.
    Its running slack vanishes along an optimal trajectory, and its terminal slack is zero.
    Consequently, an optimal primal--dual pair has zero gap.

    General FOM retains a lower-bound certificate without requiring the full value function.
    Its cumulative slack budget permits candidate prefixes to be screened at finite certificate accuracy.

    DDP/SQP Hessians and CMA-ES/MPPI search covariances also restrict exploration locally (trust regions, ellipsoids). But unlike certificates, they carry no global lower-bound guarantee.
\end{remark}

\subsection{Explicit FOM and its asymptotic consistency}
\label{subsec:explicit_fom_consistency}

We first consider the realization in which primal feasibility is represented \emph{explicitly} through a finite weak Liouville test family.
This is the canonical Eulerian form of FOM. The transport constraints remain visible at finite resolution.
It separates two approximation layers, finite trial expressivity on the primal side and weak-form refinement on the transport side. The certificate class remains an independent finite-dimensional realization of HJB-type dual objects.

Since enforcing the full Liouville residual functional $\mathcal R_\theta(\cdot)$ is typically intractable, primal feasibility is monitored through a finite test family
\[
\mathcal V_m=\mathrm{span}\{v_1,\dots,v_m\}\subset\mathcal V.
\]
Take $v_1=1$ and $v_2=t$, and use $m\ge2$ throughout this subsection. These tests enforce the terminal and occupation masses.
Given a finite-capacity trial family $\Theta^{(M)}$, each parameter $\theta\in\Theta^{(M)}$ induces a global primal pair $(\mu_\theta,\mu_{T,\theta})\in\mathcal M$ with realized cost $J(\theta):=\langle \ell,\mu_\theta\rangle+\langle g,\mu_{T,\theta}\rangle$.

Relative to the chosen test basis, the explicit residual vector is defined by
\begin{equation}
\label{eq:fomResidualVector}
r_\theta^{(m)}
:=
\Big(
\langle v_j,\mu_{T,\theta}-\mu_0\rangle
-
\langle L_f v_j,\mu_\theta\rangle
\Big)_{j=1}^m
\in \mathbb R^m,
\end{equation}
whose magnitude is quantified by a norm $\|\cdot\|_{\ast,m}$ on $\mathbb R^m$.

This yields the explicit restricted feasible family and its associated value
\[ 
\begin{aligned}
&\mathcal F^{\mathrm{exp}}_{m,M}(\eta) := \Bigl\{ \theta\in\Theta^{(M)} :\; \|r_\theta^{(m)}\|_{\ast,m}\le \eta \Bigr\}, \\
&P^{\mathrm{exp}}_{m,M}(\eta) := \inf_{\theta\in\mathcal F^{\mathrm{exp}}_{m,M}(\eta)} J(\theta). 
\end{aligned}
\] 
In the exact measure space, their infinite-dimensional counterparts are the $m$-truncated feasible set and its corresponding value 
\[
\begin{aligned}
&\mathcal F_m := \bigl\{ (\mu,\mu_T)\in \mathcal M :\\
&\qquad\langle v,\mu_T-\mu_0\rangle=\langle L_f v,\mu\rangle,\ \forall v\in\mathcal V_m \bigr\}, \\
&P_m := \inf_{(\mu,\mu_T)\in\mathcal F_m} J(\mu,\mu_T).
\end{aligned}
\]

The explicit FOM separates two sources of error. 1. \emph{Representation approximation}: For a fixed test capacity $m$,  the finite trial family $\Theta^{(M)}$ recovers near-optimal costs in $\mathcal F_m$ as $M\to\infty$. 2. \emph{Weak-form refinement}: As the test family $\mathcal V_m$ is enriched ($m\to\infty$), the truncated Liouville formulation recovers the exact OM LP.

The test family $\mathcal V_m$ (for evaluating primal residuals) and the certificate class $\mathcal V_\Psi$ (for bounding HJB subsolutions) serve distinct roles, and they can be chosen independently.

 We use three approximation conditions:
\begin{enumerate}
    \item \textbf{Nested capacity.} $\Theta^{(M)}\subseteq\Theta^{(M+1)}$ for all $M\ge 1$. Increasing $M$ enriches the trial family without discarding previously attainable realizations.
    \item 
    \textbf{Near-optimal recovery.} For each fixed $m\ge2$ and every $a>0$, there exists a sequence $\theta_M\in\Theta^{(M)}$ such that
    \[
    \begin{aligned}
    \|r_{\theta_M}^{(m)}\|_{\ast,m}&\to0,\\
    \limsup_{M\to\infty}J(\theta_M)&\le P_m+a.
    \end{aligned}
    \]
    
    \item \textbf{Test density.} $\overline{\bigcup_{m\ge 1}\mathcal V_m}^{\,C^1}=\mathcal V$. 
\end{enumerate}
Assume the exact OM LP is feasible. Throughout, we assume $X$ and $\mathcal U$ are compact, $\ell$, $g$, $L_f v$ (for all $v\in C^1$) are continuous, and the trial-induced measures have uniformly bounded total mass. These guarantee weak-$\star$ relative compactness of the induced measures and weak-$\star$ continuity of the cost functionals.

\begin{theorem}[Asymptotic consistency of explicit FOM]
\label{thm:explicit_fom_consistency}
Under conditions~(i) and~(ii),
\[
\lim_{\eta\downarrow 0}\ \lim_{M\to\infty} P^{\mathrm{exp}}_{m,M}(\eta)=P_m
\qquad\text{for every fixed }m\ge2.
\]
If additionally condition~(iii) holds,
\[
\lim_{m\to\infty}\ \lim_{\eta\downarrow 0}\ \lim_{M\to\infty} P^{\mathrm{exp}}_{m,M}(\eta)=P.
\]
\end{theorem}

\begin{proof}
Fix $m$ and write $P_{m,M}(\eta):=P^{\mathrm{exp}}_{m,M}(\eta)$.

We first prove $\lim_{\eta\downarrow 0}\ \lim_{M\to\infty} P_{m,M}(\eta)=P_m$ by matching upper and lower bounds.

For the upper bound, fix $a>0$ and $\eta>0$. The recovery sequence in condition (ii) is feasible for $P_{m,M}(\eta)$ for all sufficiently large $M$, so
\[
\limsup_{M\to\infty}P_{m,M}(\eta)\le P_m+a.
\]
Letting $\eta\downarrow0$ and then $a\downarrow0$ gives the upper bound.

For the lower bound, set
\[
L_m:=\liminf_{\eta\downarrow 0}\ \liminf_{M\to\infty} P_{m,M}(\eta).
\]
Choose a sequence $\eta_n\downarrow 0$ such that
\[
a_n:=\liminf_{M\to\infty} P_{m,M}(\eta_n)\to L_m.
\]
For each $n$, choose $M_n\ge n$ such that
\[
P_{m,M_n}(\eta_n)\le a_n+\frac{1}{n}.
\]
By definition of the infimum, choose $\theta_n\in\Theta^{(M_n)}$ satisfying
\[
\|r_{\theta_n}^{(m)}\|_{\ast,m}\le \eta_n,
\quad
J(\theta_n)\le P_{m,M_n}(\eta_n)+\frac{1}{n}\le a_n+\frac{2}{n}.
\]
It follows that
\[
\limsup_{n\to\infty} J(\theta_n)\le L_m.
\]

By weak-$^\star$ relative compactness, after passing to a subsequence we may assume
\[
(\mu_{\theta_n},\mu_{T,\theta_n}) \rightharpoonup^\star (\bar\mu,\bar\mu_T).
\]
Because $\|r_{\theta_n}^{(m)}\|_{\ast,m}\to 0$ in the finite-dimensional space $\mathbb R^m$,
each component of $r_{\theta_n}^{(m)}$ converges to zero.
Thus, for every basis element $v_j$ of $\mathcal V_m$,
\[
\langle v_j,\mu_{T,\theta_n}-\mu_0\rangle-\langle L_f v_j,\mu_{\theta_n}\rangle\to 0.
\]
Passing to the limit by weak-$^\star$ continuity of the pairings gives
\[
\langle v_j,\bar\mu_T-\mu_0\rangle=\langle L_f v_j,\bar\mu\rangle,
\qquad j=1,\dots,m.
\]
By linearity, the same holds for every $v\in\mathcal V_m$, so
$(\bar\mu,\bar\mu_T)\in\mathcal F_m$.
Consequently, by the weak-$^\star$ continuity of the cost functional, we obtain the lower bound
\[ 
    P_m \le \langle \ell,\bar\mu\rangle+\langle g,\bar\mu_T\rangle = \lim_{n\to\infty} J(\theta_n) \le \limsup_{n\to\infty} J(\theta_n) \le L_m.
\]
Hence $P_m\le L_m$, i.e.,
\[
P_m\le \liminf_{\eta\downarrow 0}\ \liminf_{M\to\infty} P_{m,M}(\eta).
\]

Combining the upper and lower bounds yields
\[
\lim_{\eta\downarrow 0}\ \lim_{M\to\infty} P_{m,M}(\eta)=P_m.
\]
Nested capacity makes $P_{m,M}(\eta)$ nonincreasing in $M$, while decreasing $\eta$ can only increase it; hence these iterated limits exist.

For the second claim, the tests $1$ and $t$ give, on every $\mathcal F_m$,
\[
\begin{aligned}
\mu_T(X)&=\mu_0(X),\\
\mu(Z)&=(T-t_0)\mu_0(X).
\end{aligned}
\]
Thus the truncated feasible measures have uniformly bounded mass and compact support. Since the test spaces are nested, $P_m$ increases to a limit $L\le P$. Choose a pair in each $\mathcal F_m$ with cost at most $P_m+1/m$ and extract a weak-$\star$ convergent subsequence. Every test in $\bigcup_m\mathcal V_m$ is eventually enforced and passes to the limit. Condition (iii) and continuity of the Liouville functional in the $C^1$ norm extend this equality to every $v\in\mathcal V$. The limit is therefore OM-feasible, and cost continuity gives $P\le L$. Hence $P_m\uparrow P$, proving the second claim.
\end{proof}

Theorem~\ref{thm:explicit_fom_consistency} shows that convergence decouples into two independent approximations:  near-optimal recovery for a fixed test family (fixed $m$), and the subsequent refinement of the test family ($m\to\infty$).

\begin{example}[A deterministic optimal-control realization of explicit FOM]
    \label{ex:explicit_fom_deterministic_oc}
    Consider an optimal-control problem with initial density $\rho_0$.
    Parameterize the control by a feedback kernel $\pi_\theta(\cdot\mid t,x)\in\mathcal P(\mathcal U)$, inducing the averaged drift
    \[
    b_\theta(t,x):=\int_{\mathcal U} f(t,x,u)\,\pi_\theta(du\mid t,x).
    \]
    The induced density $\rho_\theta$ satisfies the Liouville equation
    $\partial_t\rho_\theta+\nabla_x\cdot(b_\theta\rho_\theta)=0$,
    $\rho_\theta(t_0,\cdot)=\rho_0$.
    Rather than solving this PDE, one directly parameterizes $\rho_\theta(t,x)$ alongside $\pi_\theta$ and enforces primal feasibility weakly, akin to a Galerkin projection.

    With finite test space $\mathcal V_m=\mathrm{span}\{v_1,\dots,v_m\}$, the weak Liouville residual projects onto this basis, yielding the residual vector $r_\theta^{(m)}\in\mathbb R^m$ of~\eqref{eq:fomResidualVector}.
    At fixed $m$, the primal step takes the penalized form
    \[
    \min_{\theta\in\Theta^{(M)}}
    \Big\{J(\theta)+\lambda\|r_\theta^{(m)}\|_{\ast,m}\Big\}.
    \]
    The time--state density $\rho_\theta$ can supply the reference measure $\nu_\theta$ for the certificate update~\eqref{eq:fomDualUpdate}, so primal information guides dual refinement without generating rollouts.
\end{example}

\begin{remark}[Relation to moment--SOS relaxations]
    \label{rem:moment_sos_explicit_fom}

    The finite-test measure problem defining $P_m$ provides a common reference for explicit FOM and moment--SOS relaxations~\cite{lasserre2008nonlinear}. Explicit FOM restricts its candidate measures to the parameterized trial family. For polynomial dynamics, costs, and tests on basic semialgebraic supports, moment--SOS expresses the objective and Liouville constraints in truncated moments. Moment and localizing matrix inequalities give an outer approximation of the representable moment cone.

    Let $P^{\mathrm{mom}}_{m,r}$ denote the moment relaxation value at an order $r$ sufficient to include the required moments. With the same $m$ Liouville equalities imposed in both formulations,
    \[
    P^{\mathrm{mom}}_{m,r}\le P_m\le P^{\mathrm{exp}}_{m,M}(0).
    \]
    The two bounds reflect the enlargement of the moment domain and the restriction of the measure family, respectively.

    For the exact OM problem, a fully Liouville-feasible FOM trial gives an upper bound, and a globally feasible SOS HJB certificate gives a lower bound. These certificates also apply to general FOM trials through the residual identity~\eqref{eq:fomPrimalDualIdentity}.
\end{remark}

\subsection{Implicit FOM and its asymptotic consistency}
\label{subsec:implicit_fom_consistency}

We next consider implicit FOM realizations based on rollouts.
This realization generates local dynamics directly on each time segment and measures global inconsistency through interface defects and local rollout residuals.
It recovers single shooting, multiple shooting, and general segmented simulation as special cases.

Fix a time-domain partition
\[
\Pi:=\{\tau_0=t_0<\tau_1<\cdots<\tau_K=T\}.
\]
For each finite capacity level $M$, let $\Theta_{\Pi}^{(M)}$ denote a finite-dimensional segmented-rollout family.
That is, each parameter $\theta$ specifies a segmentwise simulation-based realization, including the local rollout objects on each segment and their inter-segment coupling data.

For every parameter $\theta\in\Theta_{\Pi}^{(M)}$, the segmentwise realization produces a local occupation measure alongside its corresponding initial and terminal endpoint distributions:
\[
\begin{aligned}
&\mu_{\theta,k} \in \mathcal M_+([\tau_k,\tau_{k+1}]\times X\times \mathcal U), \\
&\nu_{\theta,k}^{+},\ \nu_{\theta,k+1}^{-} \in \mathcal M_+(X),
\end{aligned}
\quad k=0,\dots,K-1.
\]
Here $\nu_{\theta,k}^{-}$ and $\nu_{\theta,k}^{+}$ are the state marginals at the $k$-th interface $\tau_k$, arriving from the preceding segment and entering the succeeding segment, respectively. The first such interface is $\tau_0$, where $\nu_{\theta,0}^{+}=\mu_0$.

The induced global primal pair is then
\[
\mu_\theta:=\sum_{k=0}^{K-1}\mu_{\theta,k},
\qquad
\mu_{T,\theta}:=\delta_T\otimes \nu_{\theta,K}^{-},
\]
The realized cost is
\[
J(\theta):=\langle \ell,\mu_\theta\rangle+\langle g,\mu_{T,\theta}\rangle.
\]

For the implicit realization, the residual mechanism relies on two sources of inconsistency.
The first is the \emph{interface defect}. For each interface $\tau_k$, define
\[
d_{\theta,k}
:=
\nu_{\theta,k}^{+}-\nu_{\theta,k}^{-},
\qquad k=1,\dots,K-1.
\]
This quantifies the mismatch between adjacent segments.

The second is the \emph{local rollout residual}.
For each $v\in\mathcal V$, define
\[
\begin{aligned}
e_{\theta,k}(v)
:=\ &\langle v(\tau_{k+1},\cdot),\nu_{\theta,k+1}^{-}\rangle
-\langle v(\tau_k,\cdot),\nu_{\theta,k}^{+}\rangle\\
&-\langle L_f v,\mu_{\theta,k}\rangle,\qquad k=0,\dots,K-1.
\end{aligned}
\]
This is the weak Liouville defect of the local rollout on the $k$-th segment.
For an exact segmentwise rollout, one has $e_{\theta,k}(v)=0$ for every $v$.
For a numerical ODE integrator, this term measures the weak integration error on the
$k$-th segment.

These two objects recover the global weak residual. Defining
\[
\mathcal R_\theta(v)
:=
\langle v,\mu_{T,\theta}-\mu_0\rangle
-
\langle L_f v,\mu_\theta\rangle,
\]
we have the following decomposition.

\begin{lemma}[Segmented Liouville residual decomposition]
\label{lem:segmented_rollout_residual_decomp}
For every $\theta\in\Theta_{\Pi}^{(M)}$ and every $v\in\mathcal V$,
\[
\mathcal R_\theta(v)
=
\sum_{k=0}^{K-1} e_{\theta,k}(v)
+
\sum_{k=1}^{K-1}\langle v(\tau_k,\cdot),d_{\theta,k}\rangle.
\]
\end{lemma}

\begin{proof}
By definition,
\[
\langle L_f v,\mu_\theta\rangle
=
\sum_{k=0}^{K-1}\langle L_f v,\mu_{\theta,k}\rangle.
\]
Substituting the definition of $e_{\theta,k}(v)$ gives
\[
\langle L_f v,\mu_{\theta,k}\rangle
=
\langle v(\tau_{k+1},\cdot),\nu_{\theta,k+1}^{-}\rangle
-
\langle v(\tau_k,\cdot),\nu_{\theta,k}^{+}\rangle
-
e_{\theta,k}(v).
\]
Summing over $k$ and rearranging yields
\[
\begin{aligned}
\langle L_f v,\mu_\theta\rangle
=
&\langle v,\mu_{T,\theta}\rangle
-
\langle v(t_0,\cdot),\nu_{\theta,0}^{+}\rangle\\
&-
\sum_{k=1}^{K-1}\langle v(\tau_k,\cdot),d_{\theta,k}\rangle
-
\sum_{k=0}^{K-1}e_{\theta,k}(v).
\end{aligned}
\]
Since $\nu_{\theta,0}^{+}=\mu_0$, this is the desired identity.
\end{proof}

By Lemma~\ref{lem:segmented_rollout_residual_decomp} the global weak residual accumulates local rollout residuals and interface defects. If each segment is generated by exact rollout, then
\[
e_{\theta,k}(v)\equiv 0
\qquad\forall v,\ \forall k.
\]
If, in addition, all interface defects vanish, then
\[
\mathcal R_\theta(v)=0
\qquad \forall v\in\mathcal V,
\]
hence the induced global pair is exactly OM-feasible.

Measure interface defects against the spatial $C^1$ unit ball:
\[
\|d\|_{\mathrm{def}}:=\sup_{\|\varphi\|_{C^1(X)}\le1}|\langle\varphi,d\rangle|.
\]
Use $C^1$ norms given by the sum of the suprema of the function and its first derivatives, so that $\|v(t,\cdot)\|_{C^1(X)}\le\|v\|_{C^1}$.

Define the aggregate interface defect
\[
D_\theta
:=
\sum_{k=1}^{K-1}\|d_{\theta,k}\|_{\mathrm{def}}.
\]
Likewise, define the aggregate rollout residual size by
\[
E_\theta
:=
\sup_{\|v\|_{C^1}\le 1}
\left|
\sum_{k=0}^{K-1} e_{\theta,k}(v)
\right|.
\]
For numerical rollouts at resolution $h$, use a $C^1$ reconstruction $x_{h,k}$ in $X$ with admissible control $u_{h,k}$. For a single trajectory,
\[
\mu_{h,k}(dt,dx,du)=dt\,\delta_{(x_{h,k}(t),u_{h,k}(t))}(dx,du).
\]
The chain rule then bounds the rollout residual by the integrated dynamics defect:
\[
E_h\le\sum_{k=0}^{K-1}\int_{\tau_k}^{\tau_{k+1}}
\|\dot x_{h,k}-f(t,x_{h,k},u_{h,k})\|\,dt.
\]
For nonnegative weighted trajectory ensembles, use the corresponding weighted bound. Thus vanishing integrated dynamics defects ensure $E_h\to0$.

The segmented-rollout restricted feasible family is then
\[
\mathcal F^{\mathrm{imp}}_{\Pi,M}(\eta,\delta)
:=
\Bigl\{
\theta\in\Theta_{\Pi}^{(M)}
:\;
D_\theta\le \eta,\ E_\theta\le \delta
\Bigr\},
\]
and the associated restricted value is
\[
P^{\mathrm{imp}}_{\Pi,M}(\eta,\delta)
:=
\inf_{\theta\in \mathcal F^{\mathrm{imp}}_{\Pi,M}(\eta,\delta)} J(\theta).
\]

Accordingly, implicit consistency turns on the decay of aggregated rollout residuals and interface defects.
We say that the segmented-rollout family is \emph{asymptotically compact} if every sequence
$M_n\to\infty$, $\theta_n\in\Theta_{\Pi}^{(M_n)}$ satisfying
\[
\sup_n J(\theta_n)<\infty,
\qquad
D_{\theta_n}\to 0,
\qquad
E_{\theta_n}\to 0,
\]
admits a weak-$^\star$ convergent subsequence of induced primal pairs
\[
(\mu_{\theta_n},\mu_{T,\theta_n})
\rightharpoonup^\star
(\bar\mu,\bar\mu_T).
\]
A convenient sufficient condition is that the induced measures remain supported in fixed compact sets with uniformly bounded total mass.

As in the explicit case, assume the exact OM LP is feasible and impose two approximation conditions:
\begin{enumerate}
    \item \textbf{Nested capacity.} $\Theta_{\Pi}^{(M)}\subseteq\Theta_{\Pi}^{(M+1)}$ for all $M\ge 1$. Increasing $M$ enriches the segmented-rollout family without discarding previously attainable realizations.
    \item 
    \textbf{Near-optimal recovery.} For every $a>0$, there exists a sequence $\theta_M\in\Theta_\Pi^{(M)}$ such that
    \[
    \begin{gathered}
    D_{\theta_M}\to0,\qquad E_{\theta_M}\to0,\\
    \limsup_{M\to\infty}J(\theta_M)\le P+a.
    \end{gathered}
    \]
\end{enumerate}
As in the explicit case, we assume $X$ and $\mathcal U$ are compact, $\ell$, $g$, $L_f v$ (for all $v\in\mathcal V$) are continuous, and the trial-induced measures have uniformly bounded total mass. These guarantee asymptotic compactness and weak-$\star$ continuity of the cost functionals.

\begin{theorem}[Asymptotic consistency of implicit FOM]
\label{thm:implicit_fom_consistency}
Under conditions~(i) and~(ii),
\[
\lim_{\delta\downarrow 0}\ \lim_{\eta\downarrow 0}\ \lim_{M\to\infty}
P^{\mathrm{imp}}_{\Pi,M}(\eta,\delta)
=
P.
\]
\end{theorem}
    
\begin{proof}
 For the upper bound, fix $a>0$ and positive $\eta,\delta$. The recovery sequence in condition (ii) eventually satisfies both residual tolerances, giving $\limsup_{M\to\infty}P^{\mathrm{imp}}_{\Pi,M}(\eta,\delta)\le P+a$. Letting $\eta$, $\delta$, and then $a$ tend to zero gives the upper bound. Nested capacity and monotonicity in the tolerances ensure existence of the iterated limits.

 For the lower bound, choose near-minimizers along a diagonal sequence $M_n\to\infty$, $\eta_n,\delta_n\downarrow0$, whose costs approach the iterated lower limit. Then $D_{\theta_n}\to0$ and $E_{\theta_n}\to0$. By asymptotic compactness, a subsequence converges weak-$\star$ to some $(\bar\mu,\bar\mu_T)$. Lemma~\ref{lem:segmented_rollout_residual_decomp} gives
\[
\begin{gathered}
|\mathcal R_{\theta_n}(v)|\le\|v\|_{C^1}(E_{\theta_n}+D_{\theta_n})\to0.
\end{gathered}
\]
for every $v\in C^1$, hence $\mathcal R_{\bar\mu,\bar\mu_T}(v)=0$ in the limit and $(\bar\mu,\bar\mu_T)$ is exactly OM-feasible. Weak-$\star$ continuity of the cost yields $P\le \lim_{n\to\infty} J(\theta_n)$, the lower bound.

Combining the bounds gives the claimed limit.
\end{proof}

Where Theorem~\ref{thm:explicit_fom_consistency} proceeds through intermediate test-family truncations, the implicit FOM obtains exact feasibility of the limit directly from vanishing defects, via Lemma~\ref{lem:segmented_rollout_residual_decomp}.

A useful special case is exact segmentwise rollout, for which $E_\theta\equiv 0$ and the consistency statement reduces to
\[
\lim_{\eta\downarrow 0}\ \lim_{M\to\infty}
P^{\mathrm{imp}}_{\Pi,M}(\eta,0)
=
P.
\]
Single shooting corresponds to $K=1$, so $D_\theta\equiv 0$ as well.
These simplifications eliminate parts of the residual mechanism; near-optimal recovery is required within the chosen rollout family. Exact rollout and single shooting alone are insufficient for asymptotic consistency.

\begin{example}[A deterministic optimal-control realization of implicit FOM]
    \label{rem:implicit_fom_practice}
    Implicit FOM covers sampling-based trajectory optimization and policy search.
    Consider a single-segment realization ($K=1$) where the control signal $u_\xi(\cdot)$ is drawn from a parameterized search distribution $q_\theta \in \mathcal P(\Xi)$. 
    Given an initial probability measure $\mu_0$, the ODE dynamics push forward $(x_0,\xi)\sim \mu_0\otimes q_\theta$ to induce the rollout cost in expectation form:
    \[
    \begin{aligned}
        J(\theta)
        &=
        \langle \ell,\mu_\theta\rangle+\langle g,\mu_{T,\theta}\rangle\\
        &=
        \mathbb E_{\substack{x_0\sim\mu_0\\ \xi\sim q_\theta}}
        \!\left[
        \int_{t_0}^T \ell\bigl(t,x_{x_0,\xi}(t),u_\xi(t)\bigr)\,dt
        +
        g\bigl(x_{x_0,\xi}(T)\bigr)
        \right].
    \end{aligned}
    \]
    When $q_\theta$ collapses to a Dirac mass, this recovers deterministic single shooting; when $q_\theta$ has broad support, it represents a randomized search policy. Information-Geometric Optimization (IGO)~\cite{ollivier2017information} is a standard solver for optimizing this class of expectation-based objectives. 
    The rollout samples can supply the reference measure $\nu_\theta$ for the certificate update~\eqref{eq:fomDualUpdate}. Section~\ref{sec:experiments} uses independently constructed certificates to guide this primal realization.
\end{example}

\begin{remark}[The recovery conditions]
\label{rem:fom_density_hypothesis}

The two consistency theorems distinguish the capacity to represent near-optimal solutions from the recovery of dynamical feasibility.

Dirac representations illustrate this distinction. For a fixed initial state, a simple explicit realization represents mass in state--time--control space by
\[
\mu_\theta=\sum_{i=1}^{M}w_i\delta_{(t_i,x_i,u_i)},\qquad w_i\ge0,
\]
with adjustable weights and support points constrained by the finite Liouville tests. As the Liouville residual vanishes, $\sum_i w_i\to T-t_0$. An implicit ensemble of $M$ trajectories from the same initial state induces
\[
\mu_\theta=\sum_i p_i\int_{t_0}^{T}
\delta_{(t,x_i(t),u_i(t))}\,dt,\quad
p_i\ge0,\quad\sum_i p_i=1.
\]
Each $p_i$ weights a Dirac mass at the entire trajectory $\gamma_i=(x_i(\cdot),u_i(\cdot))$ in trajectory space. Thus $w_i$ distribute occupation mass, while $p_i$ distribute probability over trajectories.

For implicit FOM, the feasible trajectory space may contain distinct combinational branches in nonconvex problems. Multimodal search can explore these branches, while richer rollout parameterizations expand the candidates for near-optimal recovery. A single globally feasible dual certificate guides this search by providing lower-bound information across the branches.
\end{remark}

\section{Computational consequences of FOM}
\label{sec:properties}

FOM connects a global structure (the dual certificate) to trajectory optimization through a primal-dual language. This section explores the computational consequences of the certificate-first viewpoint.

\subsection{Structural complexity over state-space complexity}
\label{subsec:properties_structural_complexity}

An OM dual certificate is closely related to a dissipation certificate through $s_v=\ell+L_fv$. Requiring $s_v\ge0$ for all admissible inputs and $v(T,\cdot)\le g$ gives a lower bound on cost. A dissipation inequality, commonly used in feedback passivation, has the form $L_fv\le w$. Taking $w=-\ell$ gives $s_v\le0$. Under a fixed feedback, this bounds cost from above when $v(T,\cdot)\ge g$. In addition, the DP value function satisfies $\inf_u s_v=0$ and $v(T,\cdot)=g$.

We call such a function $v$ a \emph{global value structure}. It encodes the control objective and is linked to the dynamics $f$ and costs $(\ell,g)$ through its HJB slack $\ell+L_fv$ and terminal mismatch $g-v(T,\cdot)$. The common differential form suggests that the structure exploited in dissipation analysis can also be used to represent OM dual certificates.

Viewed this way, the curse of dimensionality is the cost of a tensor-product representation on the state-time grid. A global value structure, however, can be represented at low complexity by exploiting structure of dynamics and task.

Let $\mathcal S = \{S_k\}_{k=1}^K$ with $S_k \subseteq \{1,\dots,n\}$ be an index family, and consider the block-additive function class
\[
\mathcal V_{\mathcal S} := \{ v(t,x) = \sum_{k=1}^{K} v_k(t,x_{S_k}) \}.
\]
If each local component $v_k$ is parameterized by $r_k$ coefficients, a function in $\mathcal V_{\mathcal S}$ requires only $r_{\mathrm{tot}}=\sum_{k=1}^K r_k$ parameters, avoiding the tensor-product grid. Whether its HJB slack inherits this structure depends on $f$ and $\ell$.

For each $k$, let $N_k \supseteq S_k$ denote the state-index neighborhood on which $f_{S_k}$ and the $k$-th local residual factor depend, and $M_k$ the corresponding control-index set. Overlapping $N_k$ reflect block interconnections. The factor graph
\[
\mathcal G := \{(S_k,N_k,M_k)\}_{k=1}^K
\]
encodes the structural complexity of the decomposition.

For each block $k$, define the blockwise transport operator
\[
\begin{aligned}
&A_k : C^1([t_0,T] \times X_{S_k}) \to C(Z),\\
A_k \phi :=\ & \partial_t \phi(t,x_{S_k}) + \nabla_{x_{S_k}}\phi(t,x_{S_k})^\top f_k(t,x_{N_k},u_{M_k}),
\end{aligned}
\]
where $Z = [t_0,T] \times X \times\mathcal U$ is the admissible region.

\begin{lemma}[Block-additive decomposition of the HJB slack]
\label{lem:block_hjb_slack_factorization}
Suppose $v \in \mathcal V_{\mathcal S}$, and that the running cost and dynamics admit compatible decompositions
\[
\begin{aligned}
\ell(t,x,u) &= \sum_{k=1}^{K} \ell_k(t,x_{N_k},u_{M_k}),\\
f_{S_k}(t,x,u) &= f_k(t,x_{N_k},u_{M_k}),
\end{aligned}
\]
for every $k$. Then the HJB slack decomposes as
\[
s_v(t,x,u) = \sum_{k=1}^{K} s_k(t,x_{N_k},u_{M_k}), \qquad
s_k := \ell_k + A_k v_k.
\]
\end{lemma}

\begin{proof}
By additivity of $v$ and linearity of the transport operator,
\[
L_f v = \sum_{k=1}^{K} \bigl( \partial_t v_k + \nabla_{x_{S_k}}v_k^\top f_{S_k} \bigr)
= \sum_{k=1}^{K} A_k v_k,
\]
where the second equality uses $f_{S_k} = f_k$. Adding $\ell = \sum_k \ell_k$,
\[
s_v = \ell + L_f v = \sum_{k=1}^{K} (\ell_k + A_k v_k) = \sum_{k=1}^{K} s_k.
\]
\end{proof}

The HJB \emph{equation} does not decompose in this way: $\inf_u \sum_k s_k \neq \sum_k \inf_u s_k$ in general, so blockwise optimality does not follow from the additive slack decomposition. The factors $s_k$ remain coupled through overlapping state neighborhoods and shared controls.

Assign each factor a \emph{local supply-rate budget} $\sigma_k(t,x_{N_k},u_{M_k})$ with $s_k\ge\sigma_k$ and require $\sum_k\sigma_k\ge0$. As in dissipative interconnections, coupling terms can cancel in the sum, allowing individual $s_k$ to take either sign.

\begin{theorem}[Blockwise synthesis of OM-dual certificates]
\label{thm:blockwise_om_dual_assembly}
Let the assumptions of Lemma~\ref{lem:block_hjb_slack_factorization} hold. Suppose there exist local supply-rate budgets $\sigma_k(t,x_{N_k},u_{M_k})$ such that
\[
s_k(t,x_{N_k},u_{M_k}) \ge \sigma_k(t,x_{N_k},u_{M_k}), \quad k=1,\dots,K,
\]
and
\[
\sum_{k=1}^{K}\sigma_k(t,x_{N_k},u_{M_k}) \ge 0
\]
for all admissible $(t,x,u)$. Then the block-additive function
\[
v(t,x) = \sum_{k=1}^{K} v_k(t,x_{S_k})
\]
satisfies $s_v(t,x,u) \ge 0$. If, in addition, the terminal cost admits a compatible decomposition $g(x) = \sum_{k=1}^{K} g_k(x_{S_k})$, and the local terminal inequalities $v_k(T,x_{S_k}) \le g_k(x_{S_k})$ hold for all $k$, then $v(T,x) \le g(x)$, and hence $v$ is a globally feasible OM-dual certificate.
\end{theorem}

\begin{proof}
By Lemma~\ref{lem:block_hjb_slack_factorization} and the local budget inequalities,
\[
s_v(t,x,u) = \sum_{k=1}^{K} s_k(t,x_{N_k},u_{M_k}) \ge \sum_{k=1}^{K} \sigma_k(t,x_{N_k},u_{M_k}) \ge 0.
\]

For the terminal condition, summing the local inequalities gives
\[
v(T,x) = \sum_{k=1}^{K} v_k(T,x_{S_k}) \le \sum_{k=1}^{K} g_k(x_{S_k}) = g(x).
\]
\end{proof}

Factor graphs of this type arise naturally in passivity-based constructive nonlinear control~\cite{sepulchre2012constructive,KOKOTOVIC2001637}, where a block-additive storage function $U = \sum_k U_k(t,x_{S_k})$ is already supported on $\mathcal G$.
Inverse-optimal cost shaping~\cite{freeman1996robust} can use the local storage functions and their supply rates to construct OM-dual certificates on the same neighborhoods. The factor graph $\mathcal G$ is invariant to the direction of the HJB slack inequality.

\begin{proposition}[Factor graph invariance from passivity storage to OM dual certificate]
\label{prop:inherited_factor_graph_ioc}
Given a passivity-based design with factor graph $\mathcal G$, storage function $U = \sum_k U_k$, and local supply rates $w_k$ depending only on $(t,x_{N_k},u_{M_k})$ satisfying $A_k U_k \le w_k$ on $Z$, suppose IOC cost shaping produces running costs $\ell_k$ on the same index sets. If there exist $\phi_k\in C^1$ depending only on $(t,x_{S_k})$ with $\phi_k\ge U_k$ vanishing where $U_k=0$, such that
\[
\sum_k\bigl(\ell_k+A_k\phi_k-w_k\bigr)\ge0,
\]
then, selecting a terminal cost $g(x) = \sum_k g_k(x_{S_k})$ on $\{S_k\}$ with $g_k \ge \phi_k(T,\cdot) - U_k(T,\cdot)$ pointwise on $X_{S_k}$, the function $v = \sum_k v_k$ with $v_k := \phi_k - U_k$ is a feasible OM dual certificate with supply-rate budgets $\sigma_k$ on the same factor graph $\mathcal G$, satisfying $s_k\ge\sigma_k$, $\sum_k\sigma_k\ge0$, and $v(T,\cdot)\le g$.
\end{proposition}

\begin{proof}
Take $v_k:=\phi_k-U_k$. Set $\sigma_k:=\ell_k+A_k\phi_k-w_k$. Then
\[
s_k-\sigma_k = w_k-A_kU_k \ge 0.
\]
The same nonnegative residual certifies both local inequalities.
The assumed budget condition gives $s_v=\sum_k s_k\ge\sum_k\sigma_k\ge0$, and the choice of $g_k$ gives $v(T,\cdot)\le g$.
The component $v_k$ depends only on $(t,x_{S_k})$, while $s_k$ and $\sigma_k$ depend only on $(t,x_{N_k},u_{M_k})$. Thus the conversion introduces no new state or control dependencies and preserves the factor graph $\mathcal G$.
\end{proof}

The choice $\phi_k=\alpha U_k$, $\alpha\ge1$, gives $v_k=(\alpha-1)U_k$ under the stated conditions. For the direct choice $v_k=U_k$, write $d_k:=w_k-A_kU_k\ge0$. Then $s_k=w_k+\ell_k-d_k$, so $\ell_k\ge d_k$ gives $s_k\ge\sigma_k:=w_k$. The original supply rates thus serve as dual budgets, with $s_k=w_k$ when $\ell_k=d_k$. If $\sum_k w_k=0$ and $g_k\ge U_k(T,\cdot)$, Theorem~\ref{thm:blockwise_om_dual_assembly} yields an OM-dual certificate on the same factor graph.

The next example uses the direct choice $v_k=U_k$, with the uncanceled supply passed to an upper-level component.

\begin{example}[Backstepping certificates for strict-feedback systems with an open port]
\label{ex:properties_strict_feedback_backstepping}
Consider an $n$-dimensional strict-feedback subsystem whose top layer is open. The control is split as $u = \alpha_n(\bar x_n) + \nu$, where $\alpha_n$ is the backstepping nominal feedback and $\nu$ is the port input driven by an upper-level decision-maker (e.g., an MPC, or game-theoretic planner). The dynamics are
\[
\begin{aligned}
\dot x_k &= f_k(\bar x_k) + g_k(\bar x_k)x_{k+1},\; k=1,\dots,n-1,\\
\dot x_n &= f_n(\bar x_n) + g_n(\bar x_n)(\alpha_n + \nu),
\end{aligned}
\]
with $\bar x_k := (x_1,\dots,x_k)$. Backstepping introduces error coordinates $z_1 = x_1$, $z_k = x_k - \alpha_{k-1}(\bar x_{k-1})$ and the block-additive storage function
\[
U = \frac12 \sum_{k=1}^n z_k^2 =: \sum_{k=1}^n U_k(z_k), \qquad U_k(z_k) = \tfrac12 z_k^2.
\]

In the $z$-coordinates, $U_k$ depends only on the scalar $z_k$; hence $S_k = \{k\}$. The error dynamics
\[
\begin{aligned}
\dot z_k &= -c_k z_k - g_{k-1}z_{k-1} + g_k z_{k+1}, \qquad k=1,\dots,n-1,\\
\dot z_n &= -c_n z_n - g_{n-1}z_{n-1} + g_n \nu,
\end{aligned}
\]
with $g_0 := 0$, $c_k > 0$, exhibit the pairwise coupling $g_k z_k z_{k+1}$ that drives the supply-rate structure; the port input $\nu$ enters only through the last equation.

With $\ell_k=c_kz_k^2$ and $v_k=U_k$, the blockwise HJB slacks become
\[
\begin{aligned}
s_k &:= A_k U_k + \ell_k \\&= g_k z_k z_{k+1} - g_{k-1}z_{k-1}z_k, \qquad k=1,\dots,n-1,\\
s_n &= g_n z_n \nu - g_{n-1}z_{n-1}z_n,
\end{aligned}
\]
where the running cost cancels the terms $-c_kz_k^2$ in $A_kU_k$. Setting $\sigma_k := s_k$, the pairwise cancellation of the cross terms leaves
\[
\sum_{k=1}^n \sigma_k = g_n z_n \nu.
\]

Let an upper-level certificate component $v_{\mathrm{upper}}$ and running cost $\ell_{\mathrm{upper}}$ satisfy $\ell_{\mathrm{upper}}+L_fv_{\mathrm{upper}}\ge\sigma_{\mathrm{upper}}$ throughout the admissible region. For $v=U+v_{\mathrm{upper}}$, the supply-rate condition of Theorem~\ref{thm:blockwise_om_dual_assembly} reduces to the scalar interface inequality
\[
\sigma_{\mathrm{upper}} + g_n z_n\,\nu \ge 0
\]
for all admissible states and inputs. Together with $U(T,\cdot)+v_{\mathrm{upper}}(T,\cdot)\le g$, this certifies the interconnected problem. The internal identities $s_k=\sigma_k$ hold blockwise, with each neighborhood $N_k$ including the state dependencies of the gains.

The same supply rate $g_n z_n \nu$ governs the primal side. Summing $\ell_k + A_k U_k = s_k$ over $k$ and using $\sum_k s_k = g_n z_n \nu$ gives
\[
\sum_{k=1}^n \ell_k = g_n z_n \nu - \dot U.
\]
Along any trajectory,
\[
\int_{t_0}^T \sum_k \ell_k\,dt = \int_{t_0}^T g_n z_n \nu\,dt + U(x(t_0)) - U(x(T)).
\]
Up to the constant $U(x(t_0))$, the subsystem contribution to the upper-level primal objective is a marginal running cost $g_n z_n \nu$ and a terminal correction $-U(x(T))$.
Any upper-level running cost and the original terminal cost are retained. The supply rate $g_n z_n \nu$ appears on both sides. It determines dual feasibility and becomes the marginal running cost seen by the upper-level planner.

\end{example}

The local decomposition also identifies components that can be omitted from the certificate.

\begin{proposition}[Certificate reduction via dissipative components]
\label{prop:properties_dissipative_reduction}
Let $v(t,x) = \sum_{k=1}^{K} v_k(t,x_{S_k})$ be $(\varepsilon,\varepsilon_T)$-feasible for the running and terminal costs $(\ell,g)$ under the compatible decomposition of Lemma~\ref{lem:block_hjb_slack_factorization}.
Suppose that for some index $j$,
\begin{equation}\label{eq:properties_dissipative_condition}
L_f v_j \le \delta \quad \text{on } Z,
\qquad
v_j(T,\cdot) \ge -\delta_T \quad \text{on } X_{S_j},
\end{equation}
with $\delta, \delta_T \ge 0$.
Then the reduced candidate $\tilde v = \sum_{k \neq j} v_k$ is $(\varepsilon+\delta,\; \varepsilon_T+\delta_T)$-feasible.
When $\delta = \delta_T = 0$, the original tolerances $(\varepsilon,\varepsilon_T)$ are preserved.
\end{proposition}

\begin{proof}
By linearity of the transport operator,
\[
\ell + L_f \tilde v
= \ell + L_f(v - v_j)
= (\ell + L_f v) - L_f v_j
\ge -\varepsilon - \delta,
\]
using $(\varepsilon,\varepsilon_T)$-feasibility of $v$ and the hypothesis $L_f v_j \le \delta$.
For the terminal condition,
\[
g - \tilde v(T,\cdot)
= g - v(T,\cdot) + v_j(T,\cdot)
\ge -\varepsilon_T - \delta_T,
\]
using $v(T,\cdot) \le g + \varepsilon_T$ and $v_j(T,\cdot) \ge -\delta_T$.
\end{proof}

\begin{remark}[Supply-rate interpretation of dissipative reduction]
\label{rem:properties_dissipative_reduction_budget}
Condition $L_f v_j \le 0$ yields $s_j \le \ell_j$.
In the supply-rate framework of Theorem~\ref{thm:blockwise_om_dual_assembly}, setting $v_j=0$ while retaining $\ell_j$ increases its slack from $s_j$ to $\ell_j$, so the original supply-rate budget remains valid.

\end{remark}

\begin{remark}[Connection to passivity]
\label{rem:properties_dissipative_reduction_passivity}
The hypothesis $L_f v_j \le 0$ is the defining inequality of a dissipation certificate with zero supply rate.
For large-scale systems, exploring feedback passivation to create passive inner loops is beneficial.
Enforcing nonpositive net supply to each selected inner loop for all admissible outer inputs makes its storage component nonincreasing in the resulting closed-loop model. These components can then be removed from the certificate while retaining the running costs.
Design effort then concentrates on the blocks for which $L_fv_k\le0$ cannot be guaranteed, where cross-term cancellations sustain feasibility.
\end{remark}

Blockwise approximation yields the following additive error propagation guarantee.

\begin{corollary}[Additive error propagation in blockwise assembly]
\label{cor:properties_approx_block_additive}
Let $v^\star = \sum_{k=1}^K v_k^\star$ be a globally feasible block-additive OM-dual certificate. Suppose that, for every $\eta > 0$, each block admits an approximation $\tilde v_{k,\eta} \in \mathcal U_k^{(r_k(\eta))}$ satisfying
\[
\sup_{Z} |A_k(\tilde v_{k,\eta} - v_k^\star)| \le \eta, \qquad
\sup_{X_{S_k}} |\tilde v_{k,\eta}(T,\cdot) - v_k^\star(T,\cdot)| \le \eta.
\]
Then $\tilde v_\eta := \sum_{k=1}^K \tilde v_{k,\eta}$ is $(K\eta, K\eta)$-feasible. Hence it yields the certified lower bound
\[
\langle \tilde v_\eta(t_0,\cdot), \mu_0\rangle
- (T-t_0)\mu_0(X)K\eta
- \mu_0(X)K\eta \le P,
\]
using $r_{\mathrm{tot}}(\eta) = \sum_{k=1}^K r_k(\eta)$ coefficients.
\end{corollary}

\begin{proof}
By Lemma~\ref{lem:block_hjb_slack_factorization},
\[
\ell + L_f \tilde v_\eta = \ell + L_f v^\star + \sum_{k=1}^K A_k(\tilde v_{k,\eta} - v_k^\star).
\]
Since $v^\star$ is feasible, $\ell + L_f v^\star \ge 0$, while the last sum is bounded below by $-K\eta$. Hence $\ell + L_f \tilde v_\eta \ge -K\eta$. Similarly,
\[
g - \tilde v_\eta(T,\cdot) = g - v^\star(T,\cdot) + \sum_{k=1}^K \bigl(v_k^\star(T,\cdot) - \tilde v_{k,\eta}(T,\cdot)\bigr) \ge -K\eta.
\]
Thus $\tilde v_\eta$ is $(K\eta,K\eta)$-feasible. The lower bound follows from Proposition~\ref{prop:prelim_approx_OM_lower_bound}, and the coefficient count is immediate.
\end{proof}

Blockwise assembly (Theorem~\ref{thm:blockwise_om_dual_assembly}), factor graph inheritance (Proposition~\ref{prop:inherited_factor_graph_ioc}), dissipative reduction (Proposition~\ref{prop:properties_dissipative_reduction}), and finite-dimensional approximation (Corollary~\ref{cor:properties_approx_block_additive}) share a single factor graph.
If $r_k(\eta) \lesssim \eta^{-\beta_k}$, then
\[
r_{\mathrm{tot}}(\eta) \lesssim K\,\eta^{-\beta}, \quad \beta := \max_k \beta_k.
\]
For a prescribed global accuracy $\varepsilon$, the local tolerance is tightened to $\eta = \varepsilon/K$, yielding $r_{\mathrm{tot}}(\varepsilon) \lesssim K^{\beta+1}\,\varepsilon^{-\beta}$.
The bottleneck shifts from ambient dimension $n$ to interconnection topology.

\subsection{Certificates are reusable online objects}
\label{subsec:properties_online_reuse}

Trajectories may be fragile under horizon shifts and perturbations; certificates are not. Defined over the full state-time domain, a certificate remains valid across instances.

Write the horizon length as
\[
H:=T-t_0.
\]

\begin{proposition}[Reuse under shifts and perturbations]
\label{prop:properties_certificate_reuse}
Let $v(t,x)$ be $(\varepsilon,\varepsilon_T)$-feasible for the nominal problem with dynamics $f(t,x,u)$ and costs $(\ell(t,x,u), g(x))$ on the horizon $[0,H]$.
Define the time-shifted certificate
\[
v^\tau(t,x) := v(t-\tau,x),
\]
and the time-shifted nominal data
\[
f^\tau(t,x,u) := f(t-\tau,x,u),\qquad
\ell^\tau(t,x,u) := \ell(t-\tau,x,u).
\]
Consider a perturbed problem on the shifted horizon $[\tau,\tau+H]$ with costs and dynamics $(\tilde f,\tilde\ell,\tilde g)$ satisfying
\[
\|\tilde f-f^\tau\|_\infty\le\delta_f,\quad
\|\tilde\ell-\ell^\tau\|_\infty\le\delta_\ell,\quad
\|\tilde g-g\|_\infty\le\delta_g.
\]
Then $v^\tau$ is $(\tilde\varepsilon,\tilde\varepsilon_T)$-feasible for the perturbed problem, where
\[
\tilde\varepsilon
= \varepsilon+\delta_\ell+\|\nabla_x v\|_\infty\delta_f,\qquad
\tilde\varepsilon_T
= \varepsilon_T+\delta_g.
\]
It yields the certified lower bound
\[
\langle v^\tau(\tau,\cdot),\mu_\tau\rangle
- H\mu_\tau(X)\tilde\varepsilon
- \mu_\tau(X)\tilde\varepsilon_T
\le \tilde P,
\]
where $\mu_\tau$ is the state distribution at time $\tau$ and $\tilde P$ is the optimal value of the perturbed problem.
\end{proposition}

\begin{proof}
Define the time-shift operator $(\tau^* h)(t,\cdot) := h(t-\tau,\cdot)$.
Then $v^\tau = \tau^* v$, $f^\tau = \tau^* f$, $\ell^\tau = \tau^*\ell$.
Since $\tau^*$ is linear and commutes with $\partial_t$ and $\nabla_x$,
\[
L_{f^\tau} v^\tau + \ell^\tau
= L_{\tau^* f}(\tau^* v) + \tau^*\ell
= \tau^*(L_f v + \ell).
\]
The perturbed HJB slack then expands as
\[
\begin{aligned}
L_{\tilde f} v^\tau + \tilde\ell
&= L_{f^\tau} v^\tau + \ell^\tau
   + \nabla_x v^\tau{}^\top(\tilde f - f^\tau)
   + (\tilde\ell - \ell^\tau) \\[2pt]
&= \tau^*(L_f v + \ell)
   + \nabla_x v^\tau{}^\top(\tilde f - f^\tau)
   + (\tilde\ell - \ell^\tau) \\[2pt]
&\ge -\varepsilon - \|\nabla_x v\|_\infty\delta_f - \delta_\ell
= -\tilde\varepsilon.
\end{aligned}
\]
For the terminal condition,
\[
\tilde g - v^\tau(\tau+H,\cdot)
= g - v(H,\cdot) + (\tilde g - g)
\ge -\varepsilon_T - \delta_g
= -\tilde\varepsilon_T.
\]
The lower bound follows from Proposition~\ref{prop:prelim_approx_OM_lower_bound} applied with initial time $\tau$ and horizon $H$.
\end{proof}

Certificate reuse separates a change of time origin from a change in problem data. The shift $v^\tau(t,x)=v(t-\tau,x)$ preserves feasibility for $(f^\tau,\ell^\tau,g)$, providing a warm start on the new horizon. Deviations of the actual data $(\tilde f,\tilde\ell,\tilde g)$ from this baseline enter the perturbation tolerances additively. 

\section{Experiments}
\label{sec:experiments}

\subsection{Benchmark Problem}

\begin{figure}[!b]
    \centering
    \includegraphics[width=\linewidth]{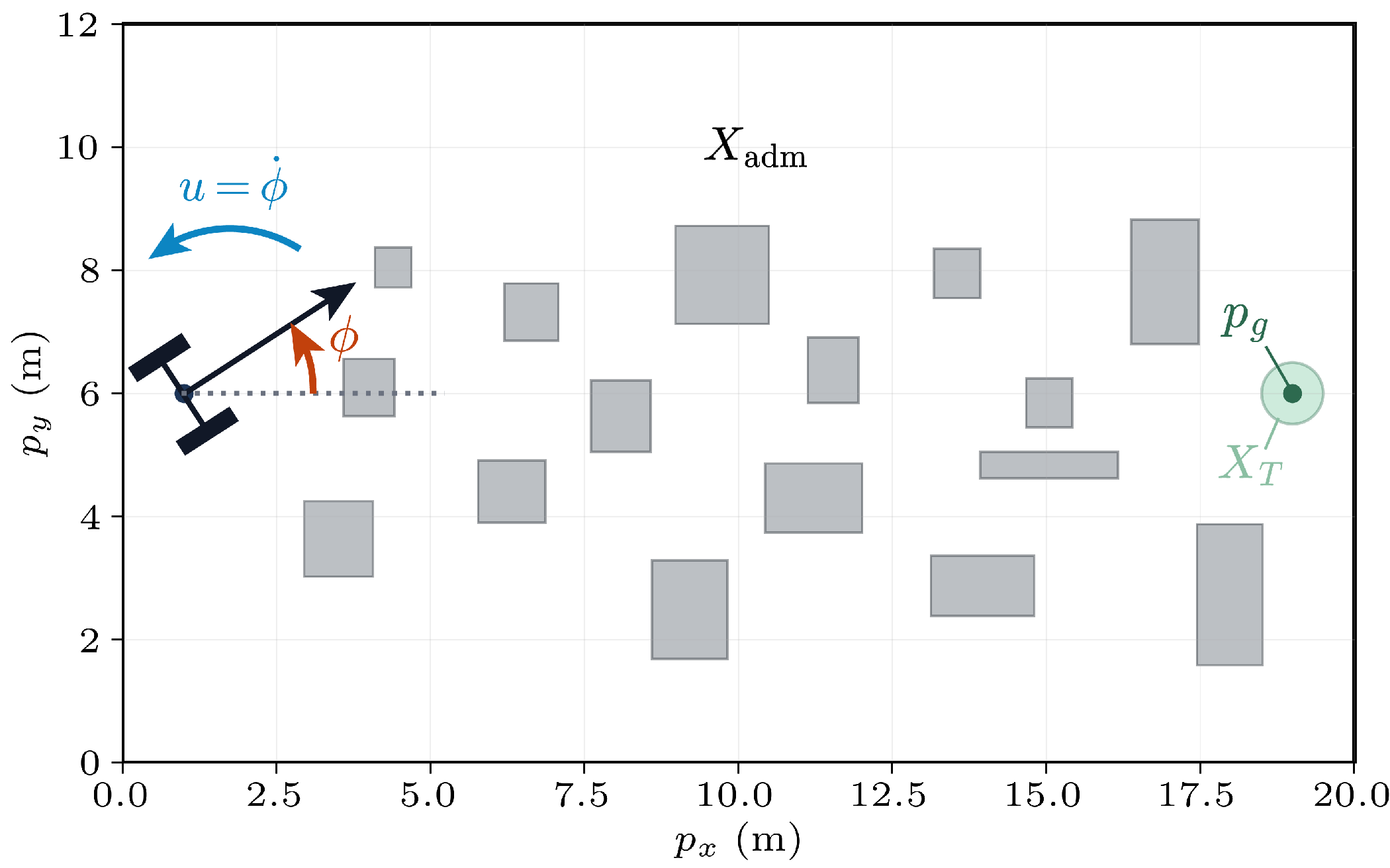}
    \caption{The benchmark problem.}
    \label{fig:benchmark_problem}
\end{figure}

\begin{figure*}[t]
    \centering
    \begin{subfigure}[t]{0.32\textwidth}
        \centering
        \includegraphics[width=\linewidth]{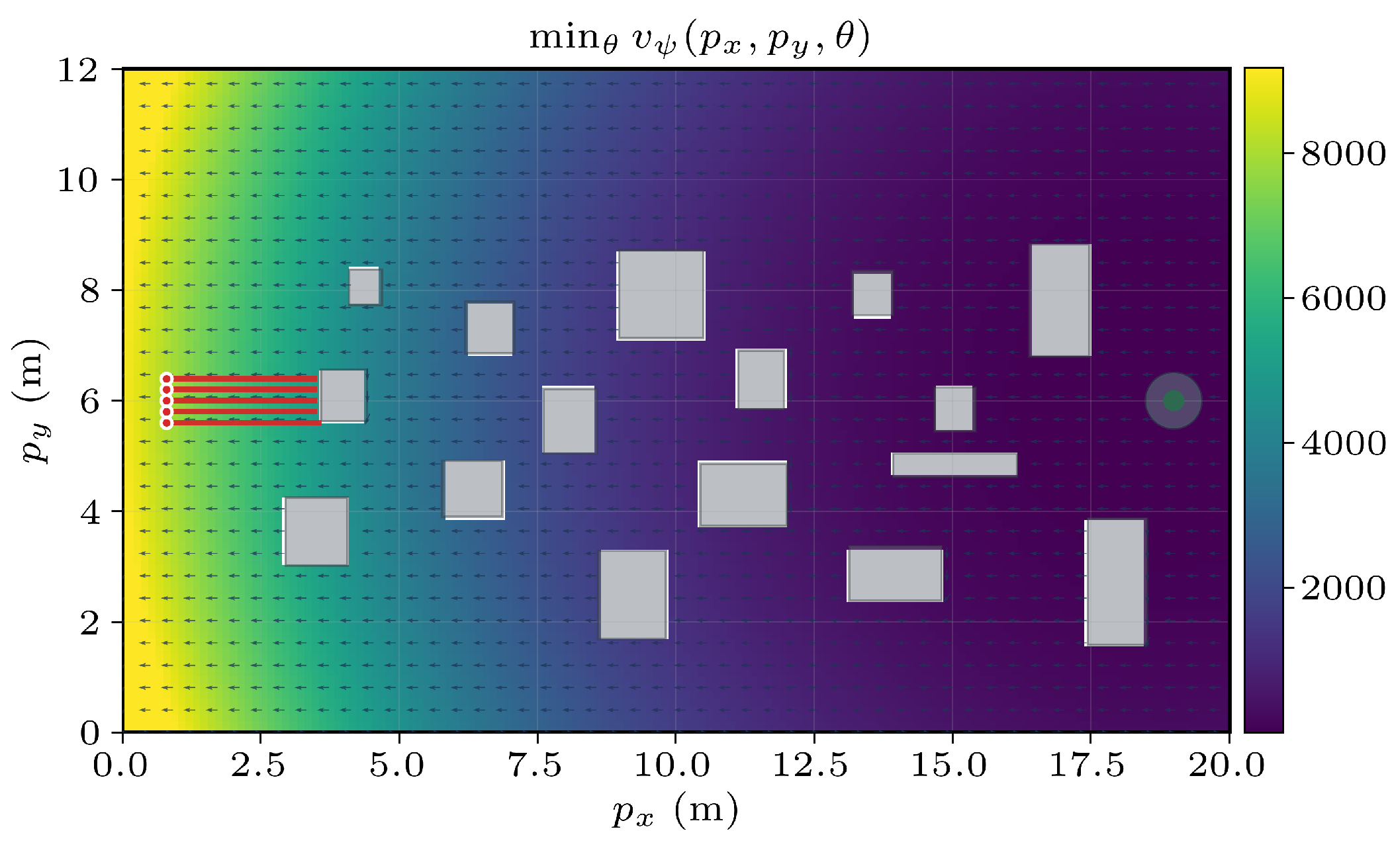}
        \caption{$K=0$ ($v_{\mathrm{rad}}(p)$).}
        \label{fig:certificate_it_00}
    \end{subfigure}\hfill
    \begin{subfigure}[t]{0.32\textwidth}
        \centering
        \includegraphics[width=\linewidth]{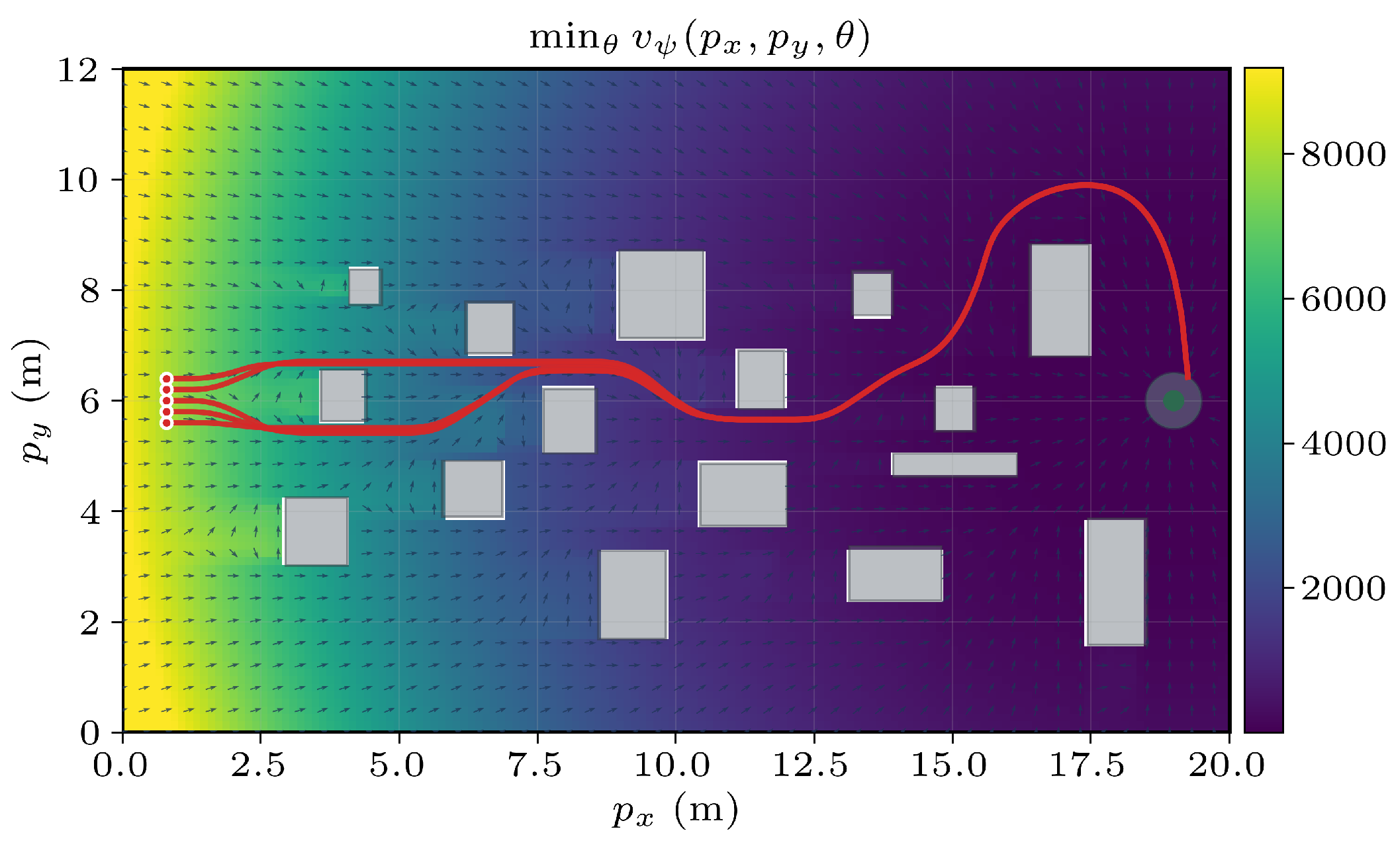}
        \caption{$K=50$.}
        \label{fig:certificate_it_50}
    \end{subfigure}\hfill
    \begin{subfigure}[t]{0.32\textwidth}
        \centering
        \includegraphics[width=\linewidth]{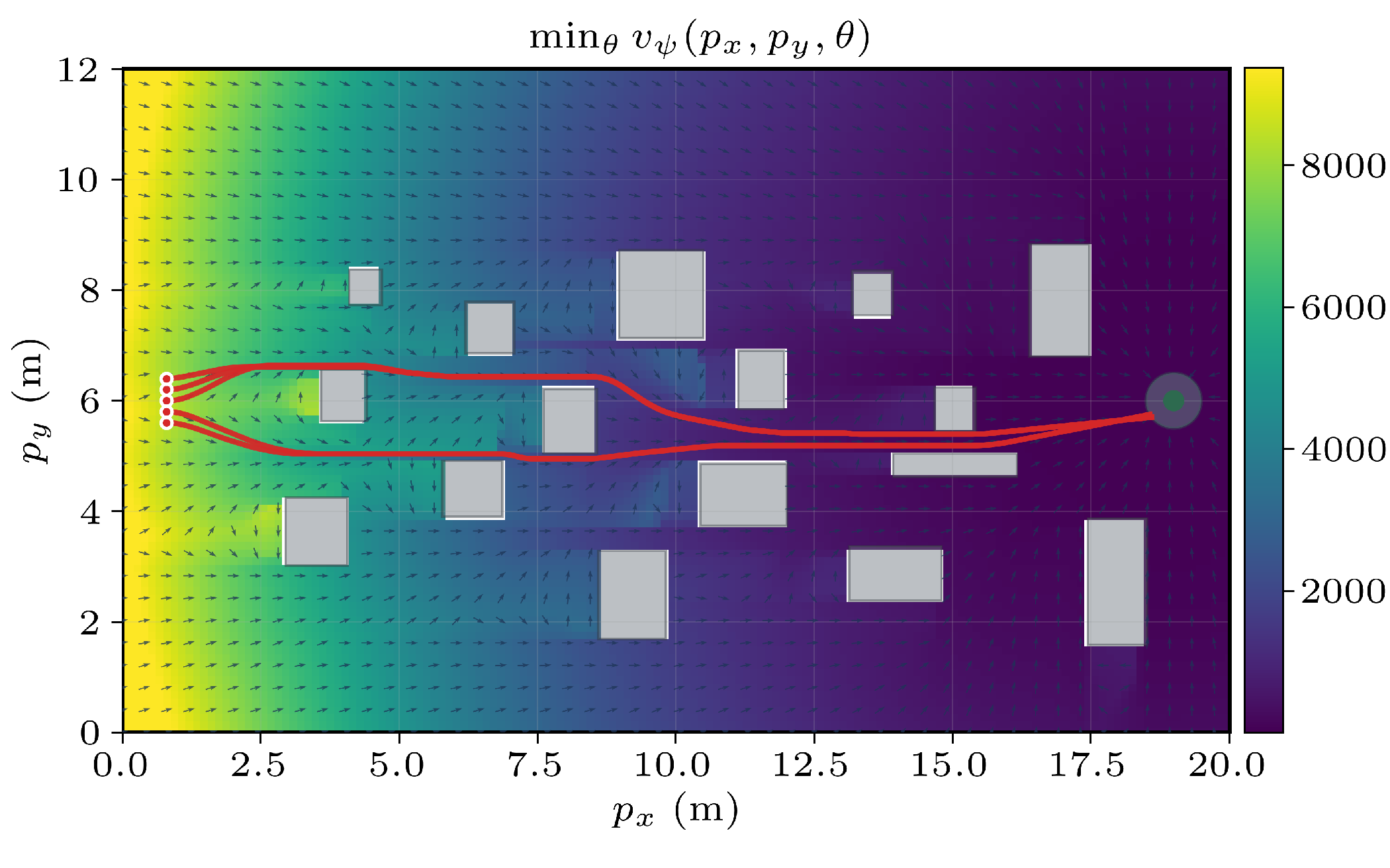}
        \caption{$K=300$.}
        \label{fig:certificate_it_300}
    \end{subfigure}
    \caption{Certificate $\min_\phi v_\psi(p_x,p_y,\phi)$ at $K=0$, $50$, and $300$, overlaid with greedy-policy rollouts (red). Additional sweeps propagate obstacle information outward from the goal. Even at $K=300$ the certificate remains coarse; it supplies global lower-bound structure, and the primal update refines the trajectories.}
    \label{fig:certificate_construction}
\end{figure*}

To obtain a benchmark that is simultaneously nonconvex, combinatorially rich, and easy to visualize, we consider a static obstacle-avoidance problem for a constant-speed unicycle. The state is
$x=[p_x,p_y,\phi]^\top\in X:=\mathbb R^2\times\mathbb S^1$,
where $(p_x,p_y)$ denotes the planar position and $\phi$ denotes the heading angle. The control is the angular velocity $u\in\mathcal U\subset\mathbb R$. The dynamics are
\[
\dot x=f(x,u),\qquad
f(x,u)=
\begin{bmatrix}
\bar{v}\cos\phi\\
\bar{v}\sin\phi\\
u
\end{bmatrix},
\]
where $\bar{v}>0$ is a constant forward speed.

Write $p=[p_x,p_y]^\top$, and let $\mathcal O_j\subset\mathbb R^2$ ($j=1,\ldots,N_{\mathrm{obs}}$) denote the static obstacle regions. The admissible state set is
\[
X_{\mathrm{adm}}:=\Big\{x\in X: p\notin\bigcup_{j=1}^{N_{\mathrm{obs}}}\mathcal O_j\Big\}.
\]
The terminal target set is defined by
\[
X_T:=\{x\in X:\ \|p-p_g\|\le \rho\},
\]
where $p_g\in\mathbb R^2$ is the goal center and $\rho>0$ is the goal radius. We choose the quadratic running cost
\[
\ell(x,u)=q_p\|p-p_g\|^2+r_u u^2,
\]
where $q_p,r_u>0$ are given parameters.
The benchmark geometry is shown in Figure~\ref{fig:benchmark_problem}.

Write $d:=\|p-p_g\|$. Define the locally Lipschitz radial lower approximation
\[
v_{\mathrm{rad}}(p):=
\begin{cases}
q_p\displaystyle\int_0^{(d-\rho)/\bar{v}}(d-\bar{v}t)^2\,dt,&d>\rho,\\
0,&d\le\rho.
\end{cases}
\]
On the exterior region $d>\rho$,
\[
v_{\mathrm{rad}}(p)=\frac{q_p}{3\bar{v}}(d^3-\rho^3),\qquad
\nabla_p v_{\mathrm{rad}}(p)=\frac{q_p}{\bar{v}}d(p-p_g).
\]
Along a unicycle trajectory $\|\dot p\|=\bar v$ and $\dot d=(p-p_g)^\top\dot p\,/\,d$, so $\dot d\ge -\bar v$. Therefore
\begin{equation}
\nabla_p v_{\mathrm{rad}}(p)\cdot\dot p
= \frac{q_p}{\bar v}d^2\dot d
\ge -q_p d^2. \label{eq:radial_gradient_bound}
\end{equation}

Now decompose the certificate $v(p_x,p_y,\phi)=v_1+v_2$ with $S_1=\{1,2\}, N_1=\{1,2,3\}$ and $S_2=\{3\}, N_2=\{3\}$. Split the running cost as $\ell_1=q_p d^2$, $\ell_2=r_u u^2$. Take \mbox{$v_2=0$}, hence \mbox{$A_2v_2=0$}, and \mbox{$v_1=v_{\mathrm{rad}}(p)$}. The zero heading component meets the component conditions of Proposition~\ref{prop:properties_dissipative_reduction} with $\delta=\delta_T=0$, and the running cost $\ell_2$ is retained. Along trajectories, $A_1v_1$ is interpreted as $\frac{d}{dt}v_{\mathrm{rad}}(p(t))\ge-q_pd(t)^2$ almost everywhere. The blockwise slacks satisfy $s_1=\ell_1+A_1v_1\ge 0$ and $s_2=\ell_2\ge 0$. Setting $\sigma_k=s_k$ yields $\sum_k\sigma_k\ge r_u u^2\ge 0$. Integration along trajectories yields the gap identity~\eqref{eq:omPrimalDualGap}, giving the lower bound for terminal cost $g=v_{\mathrm{rad}}$.

The argument is independent of the obstacle configuration; the same inequality holds on $X_{\mathrm{adm}}$. We therefore take $g(x)=v_{\mathrm{rad}}(p)$ as the terminal cost (with a slight abuse of notation $v_{\mathrm{rad}}(x)\equiv v_{\mathrm{rad}}(p)$). In the experiments, $v_{\mathrm{rad}}$ serves as both the terminal cost and the warm start for certificate tightening.

All numerical methods in the sequel are applied to a fixed-step vector-space discretization of this problem, with a time step of $\Delta t=0.05$\,s and a planning horizon of $900$ steps.

\begin{figure*}[t]
    \centering
    \begin{subfigure}[t]{0.49\textwidth}
        \centering
        \includegraphics[width=\linewidth]{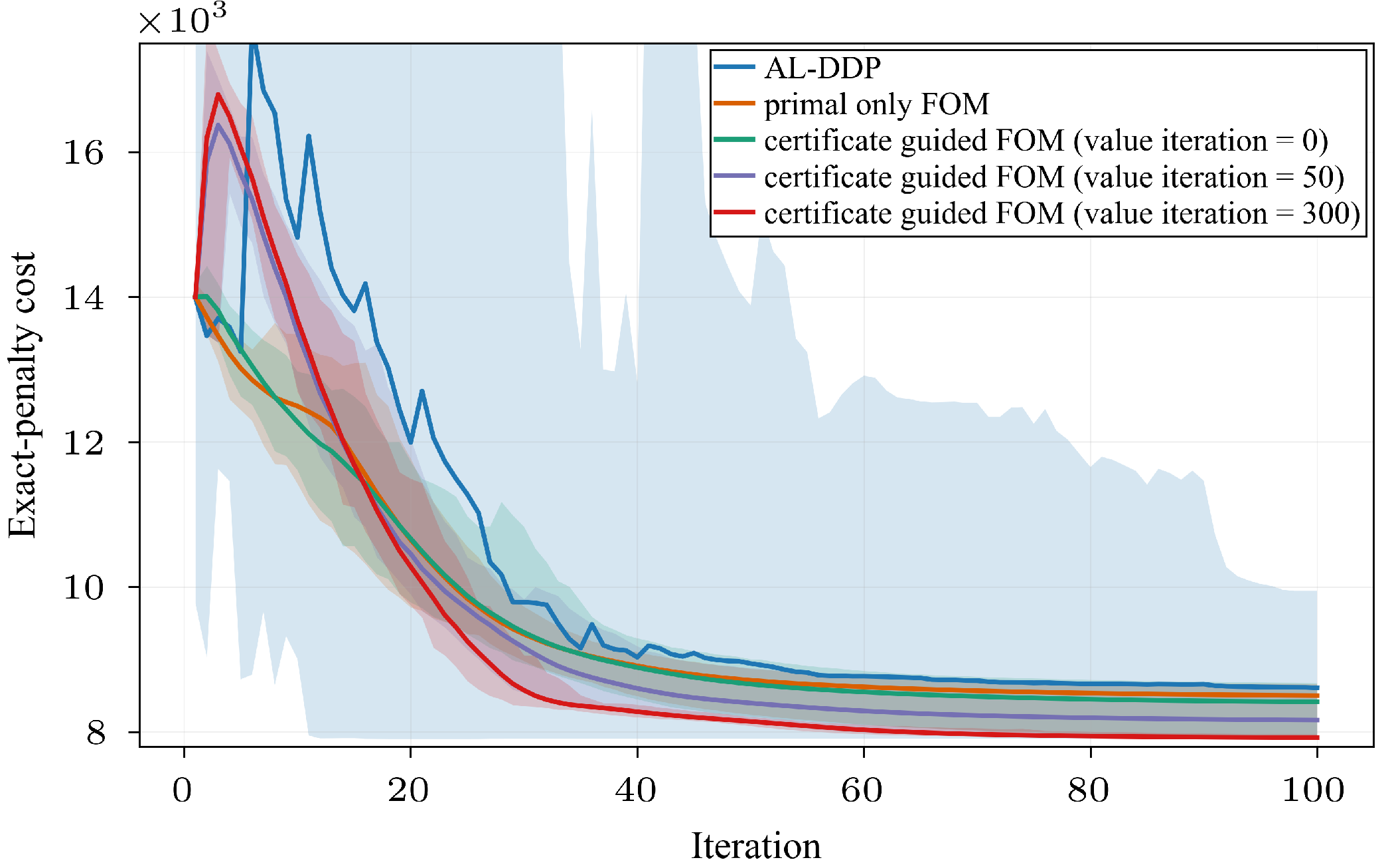}
        \caption{Exact-penalty cost over iterations.}
        \label{fig:experiments_cost}
    \end{subfigure}\hfill
    \begin{subfigure}[t]{0.49\textwidth}
        \centering
        \includegraphics[width=\linewidth]{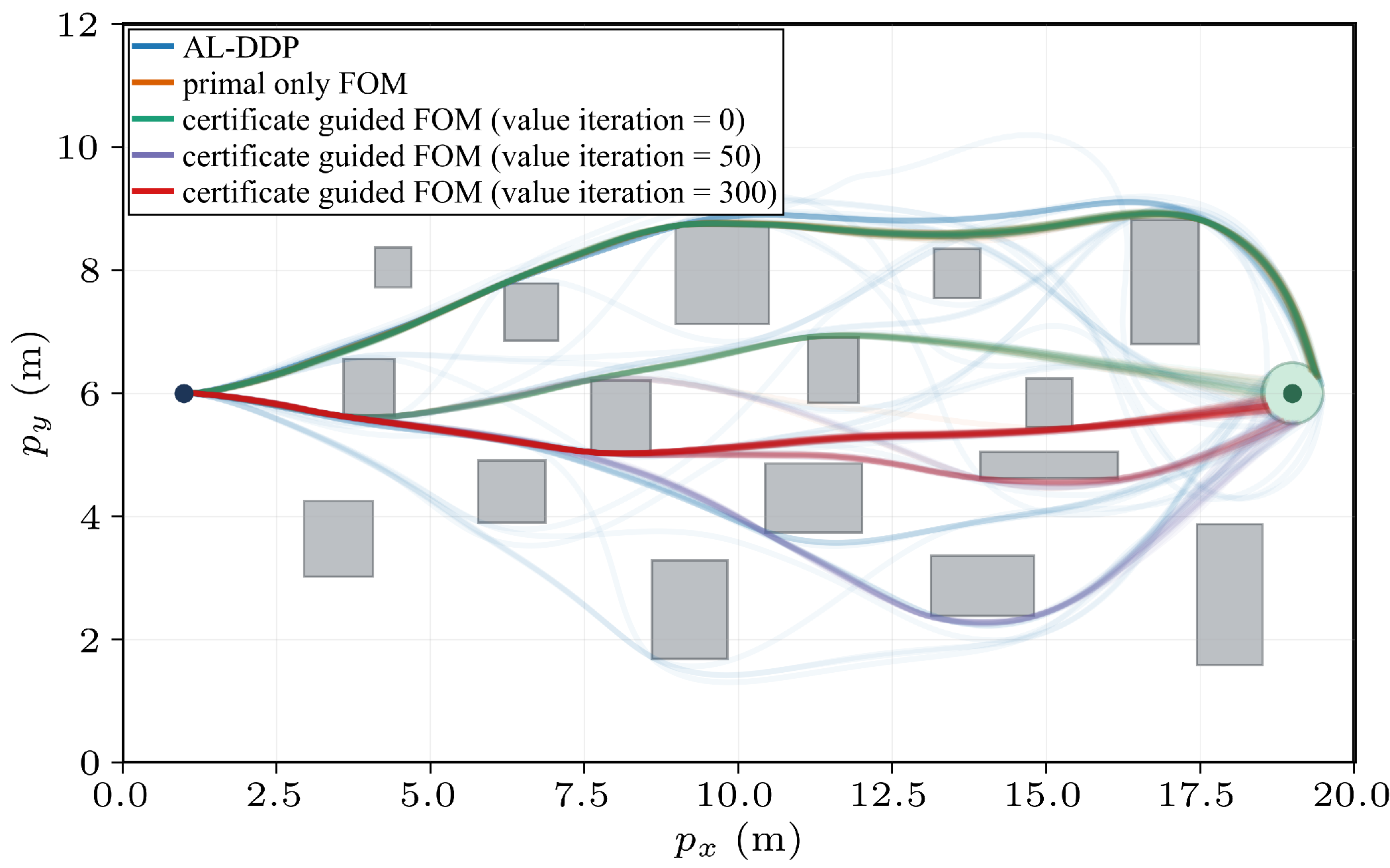}
        \caption{Best trajectory per method after $100$ iterations.}
        \label{fig:experiments_traj}
    \end{subfigure}
    \caption{Optimization behavior on the benchmark problem. In (a), thick curves denote the same mean exact-penalty cost over $50$ different random seeds for each FOM variant and over $50$ random initializations for AL-DDP; shaded regions show the corresponding minimum--maximum envelopes. In (b), tighter certificates favor the lower-cost route.}
    \label{fig:experiments_results}
\end{figure*}

\subsection{Experimental Setup}

\textbf{Certificate tightening.}
A static grid-based certificate is constructed independently of the primal search via semi-Lagrangian value iteration on a discretization of the state space $(p_x,p_y,\phi)$, initialized by $v_{\mathrm{rad}}$. Since $v_{\mathrm{rad}}$ already satisfies the discrete HJB inequality, the monotonicity of value iteration guarantees from-below improvement at every semi-Lagrangian sweep. 
The interpolated certificate candidate after $K$ iterations is used as an $(\varepsilon,\varepsilon_T)$-feasible certificate $v_\psi$, with $\varepsilon_T=\max_{X_T}(v_\psi-g)_+$ and $\varepsilon$ computed from a cellwise Lipschitz bound on the HJB slack.

We obtain three certificates of increasing tightness by halting at $K\in\{0,50,300\}$ iterations; the case $K=0$ corresponds to using $v_{\mathrm{rad}}$ directly (hence $\varepsilon=\varepsilon_T=0$), as shown in Figure~\ref{fig:certificate_construction}.
Each certificate is held fixed during the corresponding primal search.

\textbf{Primal search.}
We implement the implicit FOM of Example~\ref{rem:implicit_fom_practice} with an IGO-based search~\cite{ollivier2017information}. Let $u$ denote the discretized control vector. At iteration $k$, the search distribution is $P_{\theta_k}=\mathcal N(m_k,\Sigma_k)$ with $\Sigma_k=A_kA_k^\top$, where $\theta_k=(m_k,A_k)$. 
$\Sigma_k$ is chosen block-diagonal a priori, and $A_k$ inherits the same block structure. A sample $u\sim P_{\theta_k}$ is drawn via $u=m_k+A_k s$ with $s\sim\mathcal N(0,I)$. The primal objective is approximated by Monte Carlo rollouts:
\[
\langle \ell,\mu_{\theta_k}\rangle+\langle g,\mu_{T,\theta_k}\rangle
\approx
\mathbb E_{u\sim P_{\theta_k}}J(u).
\]
We use the IGO quantile-based rewriting and implement the resulting search step with an xNES-style natural-gradient update.
At each iteration $N=10\,000$ rollouts are evaluated and ranked by cost, with infeasible trajectories placed after all feasible ones and ordered by constraint violation, to strictly simulate exact-penalty behavior; the $N_e$ elite samples (top $5\%$ quantile) are retained and assigned equal normalized weights $\hat w_i=1/N_e$. Let $s_i$ denote their whitened vectors. The update is
\[
\begin{aligned}
m_{k+1}
&=
m_k+\eta_m A_k\sum_{i=1}^{N_e}\hat w_i s_i,\\
A_{k+1,b}
&=A_{k,b}\exp\!\left(
\frac{\eta_A}{2}\sum_{i=1}^{N_e}
\hat w_i\left(s_{i,b}s_{i,b}^\top-I_b\right)\right),
\end{aligned}
\]
where $\eta_m$ and $\eta_A$ are the step sizes for the mean and covariance factor. Here $b$ indexes the prescribed covariance blocks, $s_{i,b}$ is the corresponding subvector, and $I_b$ is the identity on that block.

\textbf{Certificate-guided pruning.}
Certificate guidance is realized through incumbent-based early termination. Let $\bar J$ denote the best feasible cost found so far. Using the admissibility envelope~\eqref{eq:fomAdmissiblePrefixes}, we test
\[
J_{\mathrm{pre}}(t)+v_\psi(x(t))-\varepsilon(T-t)-\varepsilon_T
\le \bar J+\tau_{\mathrm{tol}},
\]
where $J_{\mathrm{pre}}(t)$ is the running cost accumulated along the rollout prefix up to time $t$ and $\tau_{\mathrm{tol}}\ge 0$ is a search tolerance. Violating this bound immediately terminates and discards the rollout.

\textbf{Evaluated methods.}
We compare five methods. 
The first is augmented Lagrangian differential dynamic programming (AL-DDP), a commonly used trajectory optimizer. AL-DDP encodes obstacle constraints via smooth signed-distance fields with an augmented Lagrangian penalty and uses a restart heuristic to mitigate local-minima trapping.
The remaining four are implicit FOM variants that share the same IGO backbone. These include a primal-only variant that performs IGO search alone, and three certificate-guided variants using semi-Lagrangian certificates after $K=0$, $50$, and $300$ value iterations. The $K=0$ variant uses $v_{\mathrm{rad}}$ as its certificate. Since $v_{\mathrm{rad}}$ already serves as the terminal cost in the primal objective, this certificate provides only information the primal search already possesses. As a control, it isolates the effect of certificate-guided pruning from the effect of additional certificate structure.

\textbf{Results.}
Figure~\ref{fig:experiments_results} summarizes the optimization results. Table~\ref{tab:results} reports the exact-penalty cost statistics at iteration~$100$.

\begin{table}[t]
\centering
\caption{Exact-penalty cost at iteration~$100$.}
\label{tab:results}
\begin{tabular}{lrrrr}
\toprule
Method & Mean & SD & Min & Max \\
\midrule
AL-DDP & 8609.0 & 465.7 & 7907.8 & 9945.3 \\
primal-only FOM & 8503.4 & 296.3 & 7934.7 & 8664.2 \\
cert-guided FOM ($K{=}0$)  & 8416.0 & 343.7 & 7930.1 & 8677.9 \\
cert-guided FOM ($K{=}50$) & 8163.9 & 279.7 & 7910.2 & 8654.8 \\
cert-guided FOM ($K{=}300$) & 7922.8 &  22.0 & 7909.6 & 7970.4 \\
\bottomrule
\end{tabular}
\end{table}

Near $100$ iterations in Figure~\ref{fig:experiments_cost}, AL-DDP, primal-only FOM, and $K=0$ certificate-guided FOM reach similar mean-cost levels (Table~\ref{tab:results}), all above the costs reached with tighter certificates. Certificate guidance helps only when the certificate supplies information beyond what the primal search already possesses.

More value-iteration sweeps narrow the min–max envelope and lower the mean cost at iteration 100 (Table~\ref{tab:results}). The $K=300$ certificate takes the globally favorable route in most instances, as the trajectories in Figure~\ref{fig:experiments_traj} show.

The mean-cost curves illustrate how independently constructed certificates affect the primal search. 
Primal-only FOM and $K=0$ certificate-guided FOM show monotonically decreasing mean costs in these experiments. 
Certificate-guided FOM with nontrivial HJB information first rises in cost, then drops faster. 
The certificate prunes rollout branches that the primal-only search would otherwise explore, steering the distribution away from high-cost local minima.

\section{Conclusion}
\label{sec:conclusion}

In this paper we developed a flexible finite-dimensional numerical realization of the occupation-measure framework and, in primal--dual language, explored certificate-guided global search for numerical optimal control. FOM advocates a certificate-first approach and, through the structure of the HJB slack, establishes a deep connection to passivity-based design, enabling systematic hierarchical design as a powerful tool for high-performance numerical optimal control.

\textbf{Limitations.}
(i) FOM is organized around finite-dimensional primal realizations; certificate guidance for nonparametric sampling-based search remains to be developed; 
(ii) joint convergence of certificate tightening and nonconvex primal search remains open.

\textbf{Future work.}
Several directions remain open:
(i) further finite-dimensional FOM realizations, in particular for hybrid, switched, and contact-rich dynamics;
(ii) extensions beyond classical information structure, including stochastic formulations, partial observation, and dynamic games;
and (iii) possible analogues of certificate-guided lifting in optimal estimation and filtering.

\bibliographystyle{IEEEtran}
\bibliography{references} 

@article{lasserre2008nonlinear,
  title={Nonlinear optimal control via occupation measures and LMI-relaxations},
  author={Lasserre, Jean B and Henrion, Didier and Prieur, Christophe and Tr{\'e}lat, Emmanuel},
  journal={SIAM journal on control and optimization},
  volume={47},
  number={4},
  pages={1643--1666},
  year={2008},
  publisher={SIAM}
}

@book{sepulchre2012constructive,
  title={Constructive nonlinear control},
  author={Sepulchre, Rodolphe and Janković, Mrdjan and Kokotović, Petar V},
  year={1997},
  publisher={Springer Science \& Business Media}
}

@inproceedings{horowitz2014linear,
  title={Linear Hamilton Jacobi Bellman equations in high dimensions},
  author={Horowitz, Matanya B and Damle, Anil and Burdick, Joel W},
  booktitle={53rd IEEE Conference on Decision and Control},
  pages={5880--5887},
  year={2014},
  organization={IEEE}
}

@article{kang2017mitigating,
  title={Mitigating the curse of dimensionality: sparse grid characteristics method for optimal feedback control and HJB equations},
  author={Kang, Wei and Wilcox, Lucas C},
  journal={Computational Optimization and Applications},
  volume={68},
  number={2},
  pages={289--315},
  year={2017},
  publisher={Springer}
}

@article{bosch2000proof,
  title={A proof of a local maximum principle for optimal control problems with mixed state constraints},
  author={Bosch, P and G{\'o}mez, JA},
  journal={Rev Invest Oper Braz},
  volume={9},
  number={3},
  pages={239--262},
  year={2000}
}

@article{bellman1966dynamic,
  title={Dynamic programming},
  author={Bellman, Richard},
  journal={science},
  volume={153},
  number={3731},
  pages={34--37},
  year={1966},
  publisher={American Association for the Advancement of Science}
}

@book{bardi1997optimal,
  title={Optimal control and viscosity solutions of Hamilton-Jacobi-Bellman equations},
  author={Bardi, Martino and Dolcetta, Italo Capuzzo and others},
  volume={12},
  year={1997},
  publisher={Springer}
}

@book{fleming2006controlled,
  title={Controlled Markov processes and viscosity solutions},
  author={Fleming, Wendell H and Soner, H Mete},
  year={2006},
  publisher={Springer}
}

@book{bertsekas2012dynamic,
  title={Dynamic programming and optimal control: Volume I},
  author={Bertsekas, Dimitri},
  volume={4},
  year={2012},
  publisher={Athena scientific}
}

@book{pontryagin1963mathematical,
  title={The Mathematical Theory of Optimal Processes},
  author={L. S. Pontryagin},
  year={1963},
  publisher={John Wiley}
}

@book{betts2010practical,
  title={Practical methods for optimal control and estimation using nonlinear programming},
  author={Betts, John T},
  year={2010},
  publisher={SIAM}
}

@article{garcke2017suboptimal,
  title={Suboptimal feedback control of PDEs by solving HJB equations on adaptive sparse grids},
  author={Garcke, Jochen and Kr{\"o}ner, Axel},
  journal={Journal of Scientific Computing},
  volume={70},
  number={1},
  pages={1--28},
  year={2017},
  publisher={Springer}
}

@article{bokanowski2013adaptive,
  title={An adaptive sparse grid semi-Lagrangian scheme for first order Hamilton-Jacobi Bellman equations},
  author={Bokanowski, Olivier and Garcke, Jochen and Griebel, Michael and Klompmaker, Irene},
  journal={Journal of Scientific Computing},
  volume={55},
  number={3},
  pages={575--605},
  year={2013},
  publisher={Springer}
}

@article{dolgov2021tensor,
  title={Tensor decomposition methods for high-dimensional Hamilton--Jacobi--Bellman equations},
  author={Dolgov, Sergey and Kalise, Dante and Kunisch, Karl K},
  journal={SIAM Journal on Scientific Computing},
  volume={43},
  number={3},
  pages={A1625--A1650},
  year={2021},
  publisher={SIAM}
}

@article{darbon2016algorithms,
  title={Algorithms for overcoming the curse of dimensionality for certain Hamilton--Jacobi equations arising in control theory and elsewhere},
  author={Darbon, J{\'e}r{\^o}me and Osher, Stanley},
  journal={Research in the Mathematical Sciences},
  volume={3},
  number={1},
  pages={19},
  year={2016},
  publisher={Springer}
}

@article{chow2019algorithm,
  title={Algorithm for overcoming the curse of dimensionality for state-dependent Hamilton-Jacobi equations},
  author={Chow, Yat Tin and Darbon, J{\'e}r{\^o}me and Osher, Stanley and Yin, Wotao},
  journal={Journal of Computational Physics},
  volume={387},
  pages={376--409},
  year={2019},
  publisher={Elsevier}
}

@article{mayne1966second,
  title={A second-order gradient method for determining optimal trajectories of non-linear discrete-time systems},
  author={Mayne, David},
  journal={International Journal of Control},
  volume={3},
  number={1},
  pages={85--95},
  year={1966},
  publisher={Taylor \& Francis}
}

@article{jacobson1970differential,
  title={Differential dynamic programming},
  author={Jacobson, David H and Mayne, David Q},
  journal={Elsevier Press},
  year={1970}
}

@inproceedings{tassa2014control,
  title={Control-limited differential dynamic programming},
  author={Tassa, Yuval and Mansard, Nicolas and Todorov, Emo},
  booktitle={2014 IEEE International Conference on Robotics and Automation (ICRA)},
  pages={1168--1175},
  year={2014},
  organization={IEEE}
}

@inproceedings{li2004iterative,
  title={Iterative linear quadratic regulator design for nonlinear biological movement systems},
  author={Li, Weiwei and Todorov, Emanuel},
  booktitle={First International Conference on Informatics in Control, Automation and Robotics},
  volume={2},
  pages={222--229},
  year={2004},
  organization={SciTePress}
}

@book{bacsar1998dynamic,
  title={Dynamic noncooperative game theory},
  author={Ba{\c{s}}ar, Tamer and Olsder, Geert Jan},
  year={1998},
  publisher={SIAM}
}

@article{kappen2005linear,
  title={Linear theory for control of nonlinear stochastic systems},
  author={Kappen, Hilbert J},
  journal={Physical review letters},
  volume={95},
  number={20},
  pages={200201},
  year={2005},
  publisher={APS}
}

@article{theodorou2010generalized,
  title={A generalized path integral control approach to reinforcement learning},
  author={Theodorou, Evangelos and Buchli, Jonas and Schaal, Stefan},
  journal={The Journal of Machine Learning Research},
  volume={11},
  pages={3137--3181},
  year={2010},
  publisher={JMLR. org}
}

@article{williams2017model,
  title={Model predictive path integral control: From theory to parallel computation},
  author={Williams, Grady and Aldrich, Andrew and Theodorou, Evangelos A},
  journal={Journal of Guidance, Control, and Dynamics},
  volume={40},
  number={2},
  pages={344--357},
  year={2017},
  publisher={American Institute of Aeronautics and Astronautics}
}

@article{williams2018information,
  title={Information-theoretic model predictive control: Theory and applications to autonomous driving},
  author={Williams, Grady and Drews, Paul and Goldfain, Brian and Rehg, James M and Theodorou, Evangelos A},
  journal={IEEE Transactions on Robotics},
  volume={34},
  number={6},
  pages={1603--1622},
  year={2018},
  publisher={IEEE}
}

@article{hansen2001lao,
  title={LAO*: A heuristic search algorithm that finds solutions with loops},
  author={Hansen, Eric A and Zilberstein, Shlomo},
  journal={Artificial Intelligence},
  volume={129},
  number={1-2},
  pages={35--62},
  year={2001},
  publisher={Elsevier}
}

@article{lincoln2006relaxing,
  title={Relaxing dynamic programming},
  author={Lincoln, Bo and Rantzer, Anders},
  journal={IEEE Transactions on Automatic Control},
  volume={51},
  number={8},
  pages={1249--1260},
  year={2006},
  publisher={IEEE}
}

@article{paden2016design,
  title={Design of Admissible Heuristics for Kinodynamic Motion Planning via Sum-of-Squares Programming},
  author={Paden, Brian and Varriccho, Valerio and Frazzoli, Emilio},
  journal={arXiv preprint arXiv:1609.06277},
  year={2016}
}

@article{brown2022information,
  title={Information relaxations and duality in stochastic dynamic programs: A review and tutorial},
  author={Brown, David B and Smith, James E},
  journal={Foundations and Trends in Optimization},
  volume={5},
  number={3},
  pages={246--339},
  year={2022},
  publisher={Emerald Publishing Limited}
}

@article{vinter1993convex,
  title={Convex duality and nonlinear optimal control},
  author={Vinter, Richard},
  journal={SIAM journal on control and optimization},
  volume={31},
  number={2},
  pages={518--538},
  year={1993},
  publisher={SIAM}
}

@article{de2003linear,
  title={The linear programming approach to approximate dynamic programming},
  author={De Farias, Daniela Pucci and Van Roy, Benjamin},
  journal={Operations research},
  volume={51},
  number={6},
  pages={850--865},
  year={2003},
  publisher={INFORMS}
}

@article{kamoutsi2017infinite,
  title={On infinite linear programming and the moment approach to deterministic infinite horizon discounted optimal control problems},
  author={Kamoutsi, Angeliki and Sutter, Tobias and Esfahani, Peyman Mohajerin and Lygeros, John},
  journal={IEEE control systems letters},
  volume={1},
  number={1},
  pages={134--139},
  year={2017},
  publisher={IEEE}
}

@article{ollivier2017information,
  title={Information-geometric optimization algorithms: A unifying picture via invariance principles},
  author={Ollivier, Yann and Arnold, Ludovic and Auger, Anne and Hansen, Nikolaus},
  journal={Journal of Machine Learning Research},
  volume={18},
  number={18},
  pages={1--65},
  year={2017}
}

@article{oustry2023minimal,
  title={Minimal-time nonlinear control via semi-infinite programming},
  author={Oustry, Antoine and Tacchi, Matteo},
  journal={arXiv preprint arXiv:2307.00857},
  year={2023}
}

@book{holtorf2024bounds,
  title={Bounds and low-rank approximation for controlled Markov processes},
  author={Holtorf, Flemming},
  year={2024},
  publisher={Massachusetts Institute of Technology}
}

@article{schulman2017proximal,
  title={Proximal policy optimization algorithms},
  author={Schulman, John and Wolski, Filip and Dhariwal, Prafulla and Radford, Alec and Klimov, Oleg},
  journal={arXiv preprint arXiv:1707.06347},
  year={2017}
}

@inproceedings{silver2014deterministic,
  title={Deterministic policy gradient algorithms},
  author={Silver, David and Lever, Guy and Heess, Nicolas and Degris, Thomas and Wierstra, Daan and Riedmiller, Martin},
  booktitle={International conference on machine learning},
  pages={387--395},
  year={2014},
  organization={Pmlr}
}

@misc{varnai2020path,
  title={Path integral policy improvement: An information-geometric optimization approach},
  author={Varnai, Peter and Dimarogonas, Dimos V},
  year={2020},
  publisher={Nov}
}

@misc{cai2026qguidedsteinvariationalmodel,
      title={Q-Guided Stein Variational Model Predictive Control via RL-informed Policy Prior}, 
      author={Shizhe Cai and Zeya Yin and Jayadeep Jacob and Fabio Ramos},
      year={2026},
      eprint={2507.06625},
      archivePrefix={arXiv},
      primaryClass={cs.RO},
      url={https://arxiv.org/abs/2507.06625}, 
}

@article{posa2014direct,
  title={A direct method for trajectory optimization of rigid bodies through contact},
  author={Posa, Michael and Cantu, Cecilia and Tedrake, Russ},
  journal={The International Journal of Robotics Research},
  volume={33},
  number={1},
  pages={69--81},
  year={2014},
  publisher={Sage Publications Sage UK: London, England}
}

@inproceedings{deits2014footstep,
  title={Footstep planning on uneven terrain with mixed-integer convex optimization},
  author={Deits, Robin and Tedrake, Russ},
  booktitle={2014 IEEE-RAS international conference on humanoid robots},
  pages={279--286},
  year={2014},
  organization={IEEE}
}

@article{zhang2025simultaneous,
  title={Simultaneous trajectory optimization and contact selection for contact-rich manipulation with high-fidelity geometry},
  author={Zhang, Mengchao and Jha, Devesh K and Raghunathan, Arvind U and Hauser, Kris},
  journal={IEEE Transactions on Robotics},
  year={2025},
  publisher={IEEE}
}

@article{li2025multi,
  title={Multi-Agent Guided Policy Search for Non-Cooperative Dynamic Games},
  author={Li, Jingqi and Qu, Gechen and Choi, Jason J and Sojoudi, Somayeh and Tomlin, Claire},
  journal={arXiv preprint arXiv:2509.24226},
  year={2025}
}

@article{lidard2024blending,
  title={Blending data-driven priors in dynamic games},
  author={Lidard, Justin and Hu, Haimin and Hancock, Asher and Zhang, Zixu and Contreras, Albert Gim{\'o} and Modi, Vikash and DeCastro, Jonathan and Gopinath, Deepak and Rosman, Guy and Leonard, Naomi Ehrich and others},
  journal={arXiv preprint arXiv:2402.14174},
  year={2024}
}

@article{GRUNE202319,
title = {Examples for separable control Lyapunov functions and their neural network approximation},
journal = {IFAC-PapersOnLine},
volume = {56},
number = {1},
pages = {19-24},
year = {2023},
note = {12th IFAC Symposium on Nonlinear Control Systems NOLCOS 2022},
issn = {2405-8963},
doi = {https://doi.org/10.1016/j.ifacol.2023.02.004},
url = {https://www.sciencedirect.com/science/article/pii/S2405896323001921},
author = {Lars Grüne and Mario Sperl}
}

@INPROCEEDINGS{10383497,
  author={Sperl, Mario and Saluzzi, Luca and Grüne, Lars and Kalise, Dante},
  booktitle={2023 62nd IEEE Conference on Decision and Control (CDC)}, 
  title={Separable Approximations of Optimal Value Functions Under a Decaying Sensitivity Assumption}, 
  year={2023},
  volume={},
  number={},
  pages={259-264},
  doi={10.1109/CDC49753.2023.10383497}}

@article{KOKOTOVIC2001637,
title = {Constructive nonlinear control: a historical perspective},
journal = {Automatica},
volume = {37},
number = {5},
pages = {637-662},
year = {2001},
issn = {0005-1098},
doi = {https://doi.org/10.1016/S0005-1098(01)00002-4},
url = {https://www.sciencedirect.com/science/article/pii/S0005109801000024},
author = {Petar Kokotović and Murat Arcak}
}

@inproceedings{salzman2015asymptotically,
  title={Asymptotically-optimal motion planning using lower bounds on cost},
  author={Salzman, Oren and Halperin, Dan},
  booktitle={2015 IEEE International Conference on Robotics and Automation (ICRA)},
  pages={4167--4172},
  year={2015},
  organization={IEEE}
}

@article{hart1968formal,
  title={A formal basis for the heuristic determination of minimum cost paths},
  author={Hart, Peter E and Nilsson, Nils J and Raphael, Bertram},
  journal={IEEE transactions on Systems Science and Cybernetics},
  volume={4},
  number={2},
  pages={100--107},
  year={1968},
  publisher={IEEE}
}

@article{freeman1996robust,
  title={Inverse optimality in robust stabilization},
  author={Freeman, Randy A and Kokotovic, Petar V},
  journal={SIAM journal on control and optimization},
  volume={34},
  number={4},
  pages={1365--1391},
  year={1996},
  publisher={SIAM}
}

@book{fleming1975deterministic,
  title={Deterministic and Stochastic Optimal Control},
  author={Fleming, Wendell H and Rishel, Raymond W},
  year={1975},
  publisher={Springer}
}

@article{subbotina2006method,
  title={The method of characteristics for Hamilton—Jacobi equations and applications to dynamical optimization},
  author={Subbotina, NN},
  journal={Journal of mathematical sciences},
  volume={135},
  number={3},
  pages={2955--3091},
  year={2006},
  publisher={Springer}
}

@inproceedings{yin2024stein,
  author    = {Yin, Zeya and Lai, Tin and Khan, Subhan and Jacob, Jayadeep and Li, Yonghui and Ramos, Fabio},
  title     = {{Stein} Movement Primitives for Adaptive Multi-Modal Trajectory Generation},
  booktitle = {2024 IEEE/RSJ International Conference on Intelligent Robots and Systems (IROS)},
  year      = {2024},
  pages     = {11901--11908},
  doi       = {10.1109/IROS58592.2024.10801403}
}

@inproceedings{yin2025diverse,
  author    = {Yin, Zeya and Lai, Tin and Barcelos, Lucas and Jacob, Jayadeep and Li, Yonghui and Ramos, Fabio},
  title     = {Diverse Motion Planning with {Stein} Diffusion Trajectory Inference},
  booktitle = {2025 IEEE International Conference on Robotics and Automation (ICRA)},
  year      = {2025},
  pages     = {15610--15616},
  doi       = {10.1109/ICRA55743.2025.11127953}
}

@book{falcone2014semi,
  title     = {Semi-{L}agrangian Approximation Schemes for Linear and {H}amilton--{J}acobi Equations},
  author    = {Falcone, Maurizio and Ferretti, Roberto},
  year      = {2014},
  publisher = {SIAM},
  doi       = {10.1137/1.9781611973051},
}

@article{aine2016multi,
  author  = {Aine, Sandip and Swaminathan, Siddharth and Narayanan, Venkatraman and Hwang, Victor and Likhachev, Maxim},
  title   = {{Multi-Heuristic A*}},
  journal = {The International Journal of Robotics Research},
  volume  = {35},
  number  = {1--3},
  pages   = {224--243},
  year    = {2016},
  doi     = {10.1177/0278364915594029}
}

@inproceedings{narayanan2015improved,
  author    = {Narayanan, Venkatraman and Aine, Sandip and Likhachev, Maxim},
  title     = {Improved {Multi-Heuristic A*} for Searching with Uncalibrated Heuristics},
  booktitle = {Proceedings of the Eighth International Symposium on Combinatorial Search},
  pages     = {78--86},
  year      = {2015},
  doi       = {10.1609/socs.v6i1.18350}
}

\begin{IEEEbiography}[{\includegraphics[width=1in,height=1.25in,clip,keepaspectratio]{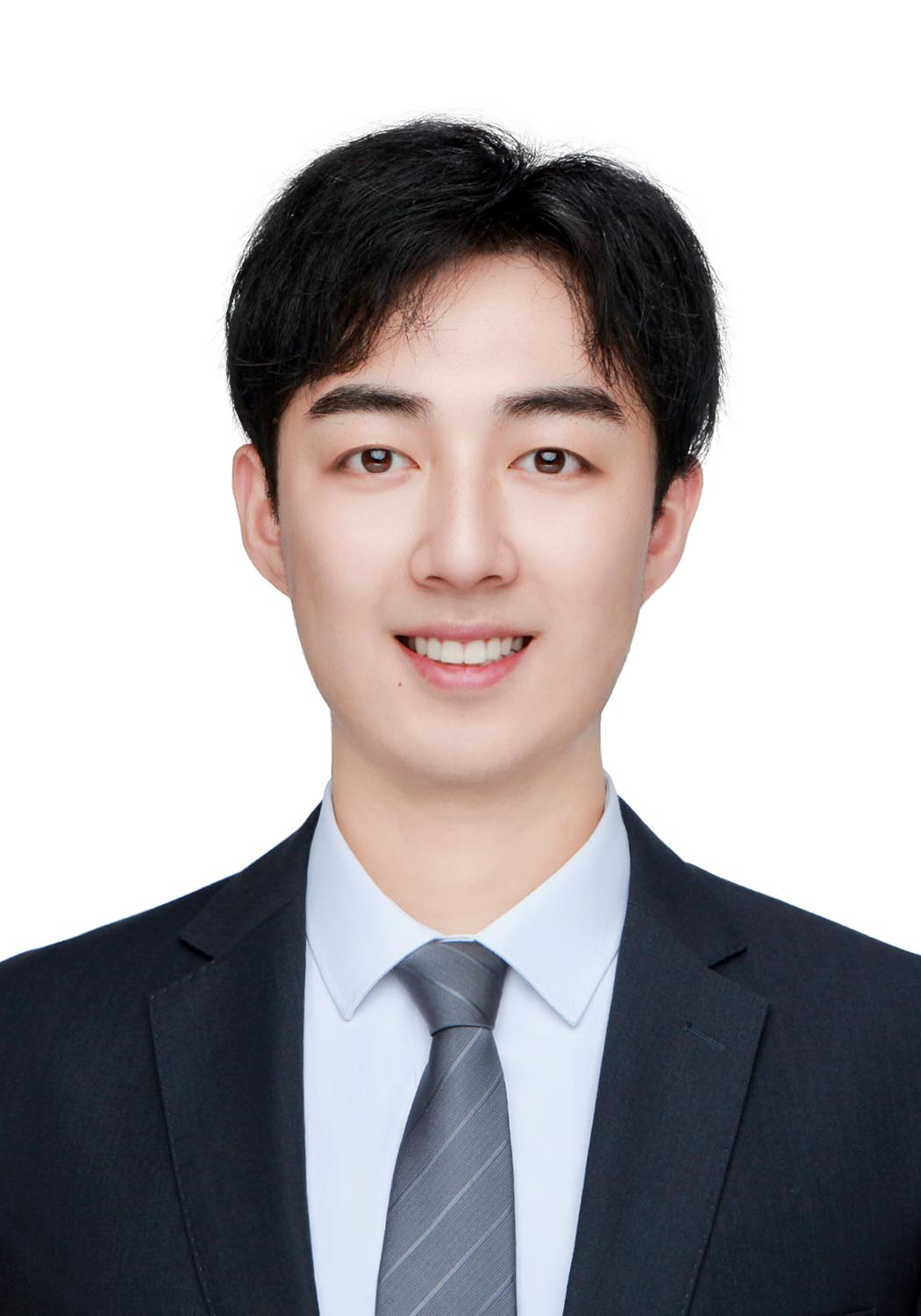}}]{Qi Wei}
received the master’s degree in mechanical engineering from Shanghai Jiao Tong University, Shanghai, in 2023.

He is currently pursuing Ph.D. degree with the School of Mechanical Engineering, Shanghai Jiao Tong University, Shanghai, China.  His research interests include modeling, nonlinear control and motion planning for heavy-duty hydraulic-driven manipulation.\end{IEEEbiography}
\vspace{-1.5\baselineskip}

\begin{IEEEbiography}[{\includegraphics[width=1in,height=1.25in,clip,keepaspectratio]{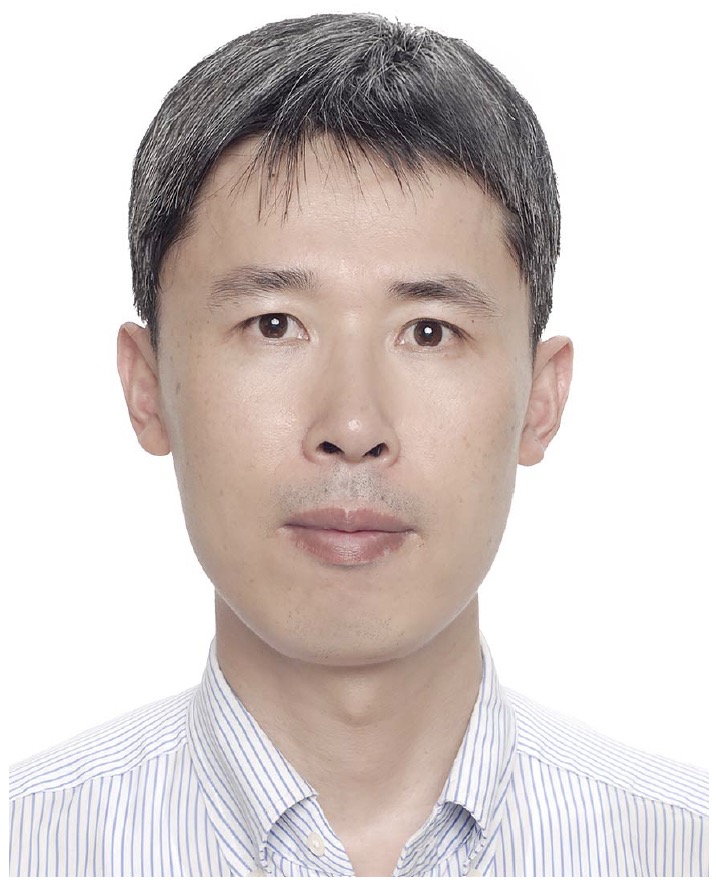}}]{Jianfeng Tao}
received the Ph.D. degree from the Beihang University, Beijing, China, in 2003.

He is currently a Professor with the Institute of Mechatronics and Logistics Equipment, School of Mechanical Engineering, Shanghai Jiao Tong University, Shanghai, China. His research interests include fluid power transmission and control, prognostics and health management, signal processing, and machine learning.\end{IEEEbiography}
\vspace{-1.5\baselineskip}

\begin{IEEEbiography}[{\includegraphics[width=1in,height=1.25in,clip,keepaspectratio]{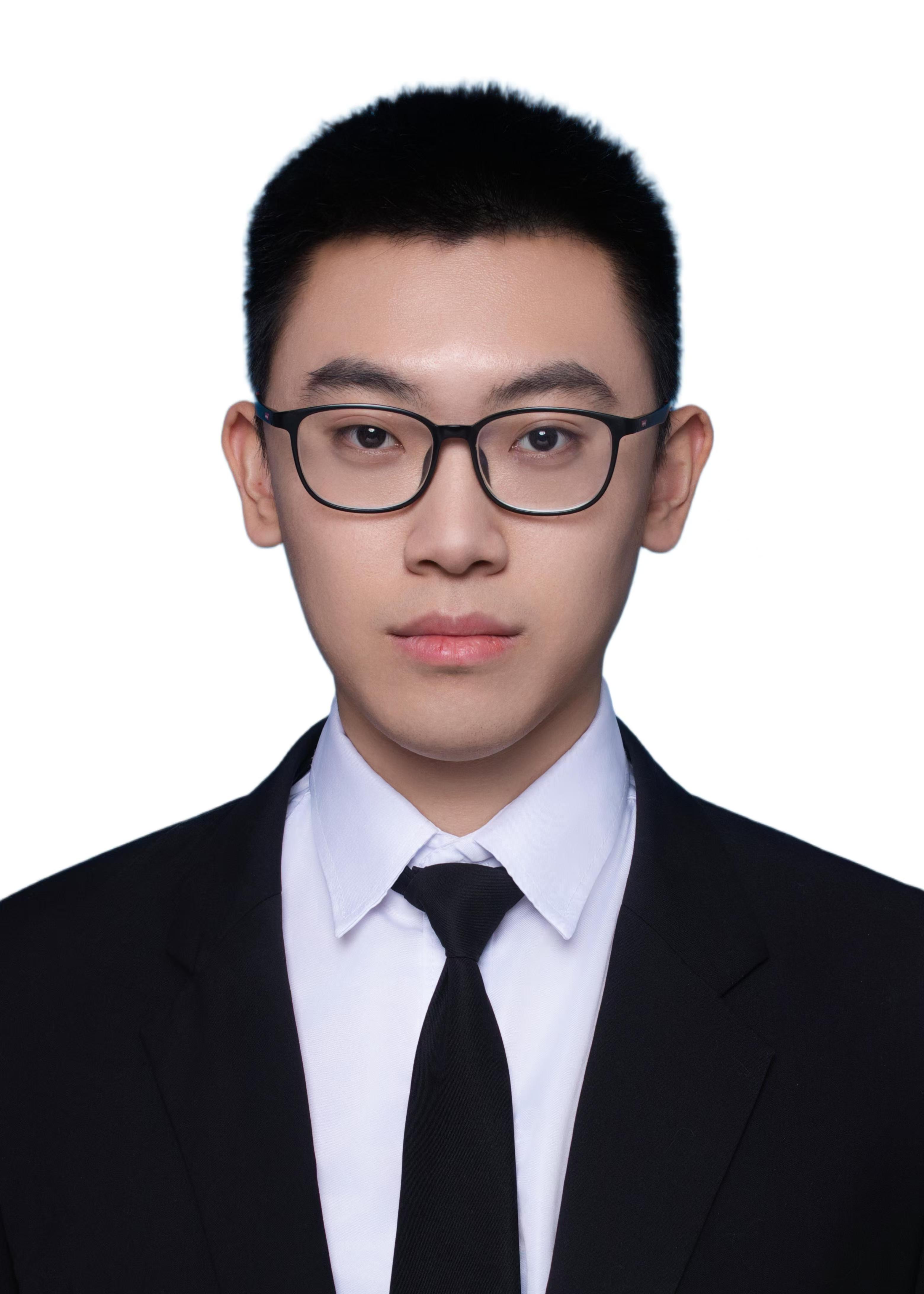}}]{Haoyang Tan}
received the master’s degree in mechanical engineering from Shanghai Jiao Tong University, Shanghai, China, in 2025.

He is currently pursuing the Ph.D. degree with the School of Mechanical Engineering, Shanghai Jiao Tong University, Shanghai, China. His research interests include performance degradation mechanisms, health monitoring, fault diagnosis, and control strategies for electro-hydraulic servo valves and fluid power systems.
\end{IEEEbiography}
\vspace{-1.5\baselineskip}

\begin{IEEEbiography}[{\includegraphics[width=1in,height=1.25in,clip,keepaspectratio]{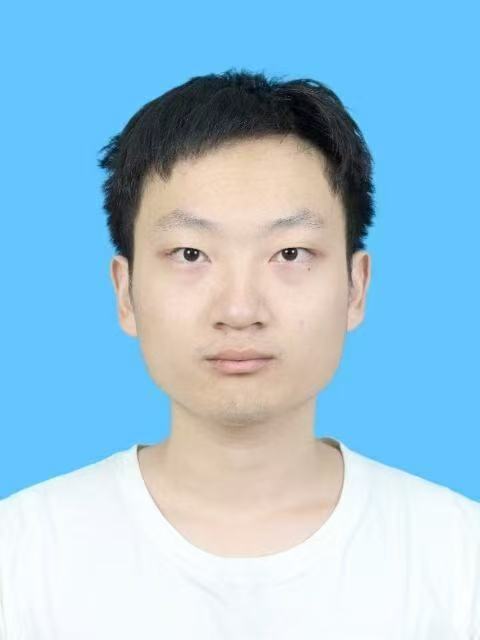}}]{Hongyu Nie}
Hongyu Nie is currently pursuing a Ph.D. degree with the School of Mechanical Engineering, Shanghai Jiao Tong University, Shanghai, China. His research interests include data-driven modeling of hydraulic systems, electro-hydraulic system control technology, and intelligent control and decision-making for heavy-duty robotic manipulators.
\end{IEEEbiography}
\vspace{-1.5\baselineskip}

\vfill

\end{document}